\pdfoutput=1 

\documentclass[11pt]{amsart} 

\usepackage{amssymb,amsthm,enumitem,colonequals,tikz-cd,microtype}
\usepackage[normalem]{ulem} 

\usepackage[osf]{Baskervaldx}
\usepackage[baskervaldx]{newtxmath}
\usepackage[cal=boondoxo]{mathalfa} 

\usepackage[top=3.75cm, bottom=3cm, left=3.5cm, right=3.5cm]{geometry}

\usepackage{xcolor}
\colorlet{darkblue}{blue!55!black}
\colorlet{darkcyan}{cyan!50!black}
\colorlet{darkgreen}{green!60!black}

\PassOptionsToPackage{hyphens}{url}
\usepackage{hyperref}
\hypersetup{
    colorlinks=true,
    linkcolor=darkblue,
    urlcolor=darkcyan,
    citecolor=darkgreen,
}

\def\eqref#1{\textcolor{darkblue}{(\ref{#1})}}

\usepackage[nameinlink]{cleveref} 
\Crefformat{section}{#2\S#1#3}
\Crefmultiformat{section}{#2\S\S#1#3}{ and~#2#1#3}{, #2#1#3}{, and~#2#1#3}

\usepackage[pagewise]{lineno}
\let\oldequation\equation
\let\oldendequation\endequation
\renewenvironment{equation}{\linenomathNonumbers\oldequation}{\oldendequation\endlinenomath}
\expandafter\let\expandafter\oldequationstar\csname equation*\endcsname
\expandafter\let\expandafter\oldendequationstar\csname endequation*\endcsname
\renewenvironment{equation*}{\linenomathNonumbers\oldequationstar}{\oldendequationstar\endlinenomath}
\let\oldalign\align
\let\oldendalign\endalign
\renewenvironment{align}{\linenomathNonumbers\oldalign}{\oldendalign\endlinenomath}
\expandafter\let\expandafter\oldalignstar\csname align*\endcsname
\expandafter\let\expandafter\oldendalignstar\csname endalign*\endcsname
\renewenvironment{align*}{\linenomathNonumbers\oldalignstar}{\oldendalignstar\endlinenomath}

\theoremstyle{plain}
\newtheorem{theorem}{Theorem}[section]
\newtheorem{lemma}[theorem]{Lemma}
\newtheorem{corollary}[theorem]{Corollary}
\newtheorem{proposition}[theorem]{Proposition}

\theoremstyle{definition}
\newtheorem{definition}[theorem]{Definition}

\newtheorem{remark}[theorem]{Remark}

\newtheorem*{ack}{Acknowledgments}

\AddToHook{env/conjecture/begin}{\crefalias{theorem}{conjecture}}
\AddToHook{env/lemma/begin}{\crefalias{theorem}{lemma}}
\AddToHook{env/corollary/begin}{\crefalias{theorem}{corollary}}
\AddToHook{env/proposition/begin}{\crefalias{theorem}{proposition}}
\AddToHook{env/definition/begin}{\crefalias{theorem}{definition}}
\AddToHook{env/remark/begin}{\crefalias{theorem}{remark}}
\AddToHook{env/setup/begin}{\crefalias{theorem}{setup}}
\AddToHook{env/example/begin}{\crefalias{theorem}{example}}
\AddToHook{env/pathology/begin}{\crefalias{theorem}{pathology}}
\AddToHook{env/conjecture/begin}{\crefalias{theorem}{conjecture}}

\numberwithin{equation}{section}
\numberwithin{theorem}{section}

\title[Fppf descent and localizing subcategories for algebraic spaces]{Fppf descent and localizing subcategories for algebraic spaces}

\author[P.~Lank]{Pat Lank}
\address{P.~Lank,
Dipartimento di Matematica “F. Enriques”, Universit\`{a} degli Studi di Milano, Via Cesare
Saldini 50, 20133 Milano, Italy}
\email{plankmathematics@gmail.com}

\date{\today}



\begin{document}
    
\begin{abstract}
    We prove descendability for fppf covers of quasi-compact quasi-separated algebraic spaces. 
    Consequently, we show that point-generatedness is detectable along fppf covers. 
    Moreover, we classify $\otimes$-localizing subcategories of the derived category of complexes with quasi-coherent cohomology on point-generated algebraic spaces in terms of subsets of the underlying topology.
\end{abstract}

\maketitle

\tableofcontents

\section{Introduction}
\label{sec:intro}

\subsection{What is known}
\label{sec:intro_what_is_known}

Neeman classified localizing subcategories of the derived category of a Noetherian ring in terms of subsets of $\operatorname{Spec}(R)$ \cite[Theorem 2.8]{Neeman:1992b}. 
Every localizing subcategory is generated by residue fields at points. 
Beyond the Noetherian setting, this classification fails. 
See \cite{Neeman:2000}.

The classification has been extended to Noetherian schemes. 
Alonso Tarr\'{i}o--Jeremi\'{a}s L\'{o}pez--Souto Salorio established a bijection between subsets of a Noetherian scheme $X$ and $\otimes$-localizing subcategories of the derived category $D_{\operatorname{qc}}(X)$ of complexes with quasi-coherent cohomology \cite[Corollary 4.13]{AlonsoTarrio/JeremiasLopez/SoutoSalorio:2004}. 
There exist localizing subcategories of $D_{\operatorname{qc}}(\mathbb{P}^1_k)$ that are not $\otimes$-compatible \cite[Example, pg.\ 595]{AlonsoTarrio/JeremiasLopez/SoutoSalorio:2004}, so the $\otimes$-condition is essential.

Stevenson \cite[Corollary 8.13]{Stevenson:2013} later recovered this classification using different methods. 
This is closely related to the tensor telescope conjecture.
It asserts that on a rigidly compactly generated triangulated category, smashing $\otimes$-Bousfield subcategories correspond to Thomason subsets of the Balmer spectrum.
The connection is that smashing $\otimes$-Bousfield subcategories are $\otimes$-localizing subcategories.

Hall--Rydh established the tensor telescope conjecture for algebraic stacks satisfying the Thomason condition \cite{Hall/Rydh:2017a}. 
Every quasi-compact quasi-separated algebraic spaces satisfies this condition, so their result describes the smashing $\otimes$-Bousfield subcategories. 
This does not classify arbitrary $\otimes$-localizing subcategories, since a $\otimes$-localizing subcategory need not be smashing. 

\subsection{What we do}
\label{sec:intro_what_we_do}

In \cite{Alonso/Jeremias/Loureiro:2024}, a quasi-compact quasi-separated scheme is called \textit{point-generated} if $D_{\operatorname{qc}}$ coincides with the localizing subcategory generated by residue fields of points. 
This motivates the following:

\begin{definition}
    \label{def:point_generated}
    A quasi-compact quasi-separated algebraic space $X$ is called \textbf{point-generated} if there exists $S\subseteq |X|$ such that $D_{\operatorname{qc}}(X) =\overline{\langle \big\{ \mathbf{R}(t_s)_\ast \mathcal{O}_{\operatorname{Spec}(k)} \big\} \rangle}$ where $t_s$ is a representative of $s\in S$.
\end{definition}

The main observation of this paper is that point-generatedness satisfies a form of fppf ascent and descent. 
To prove this, we apply descendability techniques of Balmer and Mathew (see \cite[Proposition 3.15 \& Definition 3.16]{Balmer:2016} or \cite[Definition 3.18]{Mathew:2016}).
In particular, we show descendability for fppf covers of quasi-compact quasi-separated algebraic spaces. 

\begin{proposition}
    \label{prop:fppf_descent}
    Let $g\colon Y \to X$ be an fppf cover of quasi-compact quasi-separated algebraic spaces.
    Then $\mathbf{R}g_\ast \mathcal{O}_Y$ is a descendable algebra in $D_{\operatorname{qc}}(X)$.
\end{proposition}

Variations have been proved elsewhere. 
See e.g.\ \cite[Theorem 7.1]{Hall:2022} and \cite[Theorem D]{Hall/Lamarche/Lank/Peng:2025}. 
However, we record a direct proof in \Cref{sec:descent}.

This allows us to prove that a point-generatedness of suitable algebraic space can be detected by its fppf covers. 

\begin{proposition}
    \label{prop:ascent_descent_pt_generated}
    Let $f\colon Y\to X$ be an fppf cover of quasi-compact quasi-separated algebraic spaces. 
    Then $X$ is point-generated if, and only if, $Y$ is point-generated.
\end{proposition}

As a consequence, there are many examples of point-generatedness.

\begin{corollary}
    Any Noetherian algebraic space $X$ is point-generated.
\end{corollary}

\begin{proof}
    Choose an \'{e}tale presentation $U\to X$ by an affine scheme. 
    By \cite[Theorem 2.8]{Neeman:1992b} and \cite[\href{https://stacks.math.columbia.edu/tag/071Q}{Tag 071Q}]{StacksProject}, $U$ is point-generated. 
    Thus, the claim follows from \Cref{prop:ascent_descent_pt_generated}.
\end{proof}

Combining this with the results of \cite{Alonso/Jeremias/Loureiro:2024} yields a classification:

\begin{theorem}
    \label{thm:point_generated_classification}
    Let $X$ be a point-generated algebraic space.  
    There exists a bijection
    \begin{displaymath} 
        \phi \colon 
        \big\{ \otimes\textrm{-localizing subcategories of } D_{\operatorname{qc}}(X) \big\}
        \;\longleftrightarrow\;
        \big\{ \textrm{subsets of } |X| \big\} \colon \theta
    \end{displaymath}
    given as follows: 
    \begin{itemize}
        \item $\phi(\mathcal{T})$ is the set of $p \in |X|$ represented by a $t\colon \operatorname{Spec}(k) \to X$ such that $\mathbf{R}t_\ast \mathcal{O}_{\operatorname{Spec}(k)} \in \mathcal{T}$
        \item $\theta$ sends a subset $S \subseteq |X|$ to the localizing subcategory generated by $\{\mathbf{R}(t_s)_\ast \mathcal{O}_{\operatorname{Spec}(k)}\}_{s \in S}$, where $t_s$ is a representative of $s\in S$.
    \end{itemize}
\end{theorem}

These results show that $\otimes$-localizing subcategories are generated by a set of pushforwards of structure sheaves of spectra of fields. 
Our proofs do not mimic \cite{Alonso/Jeremias/Loureiro:2024}. 
Instead, we use descendability and \'{e}tale presentations to reduce the problem to the scheme case.

\Cref{thm:point_generated_classification} extends the classification of \cite{Alonso/Jeremias/Loureiro:2024} to algebraic spaces. 
In the Noetherian case, \cite[Theorem 0.1]{Alonso/Jeremias/Loureiro:2024} can be replaced with \cite[Theorem 2.8]{Neeman:1992b}, yielding an alternative proof that proceeds from affine schemes to Noetherian algebraic spaces. 

Balchin--Omar G\'{o}mez--Stevenson have shown that the polynomial ring in countably many variables over a field is not point-generated \cite[Example 6.1]{Balchin/OmarGomez/Stevenson:2026}. 
\Cref{prop:ascent_descent_pt_generated} implies the source of every fppf cover of this scheme is likewise not point-generated.

\begin{ack}
    Lank was supported under the ERC Advanced Grant 101095900-TriCatApp. The author thanks Timothy De Deyn and Michal Hrbek for helpful discussions.
\end{ack}

\section{Preliminaries}
\label{sec:prelim}

Here, `strictly full' means a full subcategory closed under isomorphisms.

\subsection{Conventions/notation}

We follow \cite{StacksProject} for algebraic spaces and their derived categories. 
Let $X$ be a quasi-compact quasi-separated algebraic space. 
$\operatorname{Mod}(X)$ is the Grothendieck abelian category of sheaves of $\mathcal{O}_X$-modules on the small \'{e}tale site $X_{\textrm{\'{e}tale}}$ of $X$ \cite[\href{https://stacks.math.columbia.edu/tag/03LT}{Tag 03LT}]{StacksProject}. 
$\operatorname{Qcoh}(X)$ is the full subcategory of $\operatorname{Mod}(X)$ consisting of quasi-coherent sheaves. 
$D(X) := D(\operatorname{Mod}(X))$ is the derived category of $\operatorname{Mod}(X)$. 
$D_{\operatorname{qc}}(X)$ is the full subcategory of $D(X)$ consisting of complexes with quasi-coherent cohomology sheaves. 
$\operatorname{Perf}(X)$ is the full subcategory of perfect complexes in $D_{\operatorname{qc}}(X)$. 

\subsection{Generation}
\label{sec:prelim_generation}

We discuss generation for triangulated categories. 
See \cite{Bondal/VandenBergh:2003} for details. 
Let $\mathcal{T}$ be a triangulated category with shift functor $[1]\colon \mathcal{T} \to \mathcal{T}$. 
Consider a subcategory $\mathcal{S} \subseteq \mathcal{T}$. 
A triangulated subcategory of $\mathcal{T}$ is called \textbf{thick} if it is closed under direct summands. 
Denote by $\langle \mathcal{S} \rangle$ the smallest thick subcategory of $\mathcal{T}$ containing $\mathcal{S}$; if $\mathcal{S}$ consists of a single object $G$, we write $\langle G \rangle := \langle \mathcal{S} \rangle$. 
Set $\operatorname{add}(\mathcal{S})$ to be the smallest strictly full subcategory of $\mathcal{T}$ containing $\mathcal{S}$ that is closed under shifts, finite coproducts, and direct summands. 
Inductively, let $\langle \mathcal{S} \rangle_0$ consist of all objects in $\mathcal{T}$ isomorphic to the zero object, $\langle \mathcal{S} \rangle_1 := \operatorname{add}(\mathcal{S})$, and 
\begin{displaymath}
    \langle \mathcal{S} \rangle_n := \operatorname{add} \{ \operatorname{cone}(\phi) \mid \phi \in \operatorname{Hom}_{\mathcal{T}} (\langle \mathcal{S} \rangle_{n-1}, \langle \mathcal{S} \rangle_1) \}.
\end{displaymath}
It can be checked that $\langle \mathcal{S} \rangle = \cup_{n=0}^\infty \langle \mathcal{S} \rangle_n$. 

Suppose $\mathcal{T}$ admits small coproducts. 
Set $\operatorname{Add}(\mathcal{S})$ to be the smallest strictly full subcategory of $\mathcal{T}$ containing $\mathcal{S}$ that is closed under shifts, small coproducts, and direct summands. 
Inductively, let $\overline{\langle \mathcal{S} \rangle}_0$ consist of all objects in $\mathcal{T}$ isomorphic to the zero object, $\overline{\langle \mathcal{S} \rangle}_1 := \operatorname{Add}(\mathcal{S})$, and 
\begin{displaymath}
    \overline{\langle \mathcal{S} \rangle}_n := \operatorname{Add} \{ \operatorname{cone}(\phi) \mid \phi \in \operatorname{Hom}_{\mathcal{T}} (\overline{\langle \mathcal{S} \rangle}_{n-1}, \overline{\langle \mathcal{S} \rangle}_1) \}.
\end{displaymath}

If $\mathcal{T}$ admits small coproducts, then the collection of compact objects in $\mathcal{T}$ is denoted by $\mathcal{T}^c$. 
These form a triangulated subcategory of $\mathcal{T}$. 
The \textbf{localizing subcategory} generated by a collection $\mathcal{S}\subseteq \mathcal{T}$, denoted by $\overline{\langle \mathcal{S} \rangle}$, is the smallest triangulated subcategory containing $\mathcal{S}$ and closed under small coproducts. 
An induction argument on $n$ shows that $\langle \mathcal{S} \rangle_n \subseteq \overline{\langle \mathcal{S} \rangle}$ for all $n\geq 0$. 
If $\mathcal{T}^c$ is essentially small, then we say that $\mathcal{T}$ is \textbf{compactly generated} when $\mathcal{T} = \overline{\langle \mathcal{T}^c \rangle}$. 
Equivalently, $\mathcal{T}$ is compactly generated if, for any $E \in \mathcal{T}$ satisfying $\operatorname{Hom}(P, E) = 0$ for all $P \in \mathcal{T}^c$, one has $E \cong 0$.
Note that classical generators for $\mathcal{T}^c$ coincide with compact generators for $\mathcal{T}$ (see e.g.\ \cite[\href{https://stacks.math.columbia.edu/tag/09SR}{Tag 09SR}]{StacksProject}).


Assume that $\mathcal{T}$ is a tensor triangulated category with tensor $\otimes$ and unit $1$. 
The \textbf{$\otimes$-localizing subcategory} generated by a collection $\mathcal{S}\subseteq \mathcal{T}$, denoted by $\overline{\langle \mathcal{S} \rangle}_{\otimes}$, is the smallest triangulated subcategory containing $\mathcal{S}$, closed under small coproducts, and the tensor action with respect to $\mathcal{T}$. 
A useful fact, used freely in our work, is that localizing subcategories in a triangulated category with small coproducts are closed under direct summands (see \cite[Remark 1.4]{Bokstedt/Neeman:1993}). 

\subsection{Descendability}
\label{sec:prelim_descendability}

We recall a notion of descent for triangulated categories (see \cite[Definition 3.18]{Mathew:2016} for the $\infty$-categorical analog). 
Let $(\mathcal{T},\otimes,\mathbf{1})$ be a tensor triangulated category. 
Given $T\in \mathcal{T}$, denote by $\langle T \rangle_{\otimes}$ the smallest thick subcategory containing $T$ that is closed under the tensor action by $\mathcal{T}$. 
A commutative monoid $A$ in $\mathcal{T}$ is called \textbf{descendable} if $\langle A \rangle_{\otimes} = \mathcal{T}$. 
See \cite[Lemma 6.3]{DeDeyn/Lank/ManaliRahul:2025} for other characterizations. 

\begin{lemma}
    \label{lem:descent}
    Let $f\colon Y \to X$ be a morphism of quasi-compact quasi-separated algebraic spaces.
    Then $\langle \mathbf{R}f_\ast \mathcal{O}_Y \rangle_{\otimes} = \langle \mathbf{R}f_\ast D_{\operatorname{qc}}(Y) \rangle$. 
    Consequently, $\mathbf{R}f_\ast \mathcal{O}_Y$ is descendable if, and only if, $\mathcal{O}_X \in \langle \mathbf{R}f_\ast D_{\operatorname{qc}}(Y) \rangle$. 
\end{lemma}

\begin{proof}
    The second claim follows from the first by definition of descendability.
    We prove the first claim. 
    Clearly, $\mathbf{R}f_\ast \mathcal{O}_Y \in  \langle \mathbf{R}f_\ast D_{\operatorname{qc}}(Y) \rangle$. 
    Moreover, an induction argument shows $\langle \mathbf{R}f_\ast \mathcal{O}_Y \rangle_{\otimes} \subseteq \langle \mathbf{R}f_\ast D_{\operatorname{qc}}(Y) \rangle$.
    Indeed, let $E\in \langle \mathbf{R}f_\ast D_{\operatorname{qc}}(Y) \rangle_n$. 
    If $n=0$, then $E\cong 0$, and so there exists nothing to prove. 
    If $n=1$, then $E$ is a direct summand of a finite coproduct of the form $\oplus_{i\in I} \mathbf{R}f_\ast E_i$ for some $E_i\in D_{\operatorname{qc}}(Y)$, and projection formula can be used to check that $A \otimes^{\mathbf{L}} \oplus_{i\in I} \mathbf{R}f_\ast E_i \in \langle \mathbf{R}f_\ast D_{\operatorname{qc}}(Y) \rangle$ for all $A\in D_{\operatorname{qc}}(X)$. 
    Assume we have proved the claim up to $n\geq 1$. 
    Let $E\in \langle \mathbf{R}f_\ast D_{\operatorname{qc}}(Y) \rangle_{n+1}$. 
    There exists $A\in \langle \mathbf{R}f_\ast D_{\operatorname{qc}}(Y) \rangle_n$, $B\in \langle \mathbf{R}f_\ast D_{\operatorname{qc}}(Y) \rangle_1$, and a distinguished triangle 
    \begin{displaymath}
        A \to E \oplus E^\prime \to B \to A[1].
    \end{displaymath}
    Choose any $G\in D_{\operatorname{qc}}(X)$. 
    Tensoring with $G$ yields a distinguished triangle 
    \begin{displaymath}
        G\otimes^\mathbf{L} A \to G\otimes^\mathbf{L} (E \oplus E^\prime) \to G\otimes^\mathbf{L} B \to (G\otimes^\mathbf{L}A)[1].
    \end{displaymath}
    The induction hypothesis says $G\otimes^\mathbf{L} A, G\otimes^\mathbf{L} B \in \langle \mathbf{R}f_\ast D_{\operatorname{qc}}(Y) \rangle$, 
    and so,
    \begin{displaymath}
        (G\otimes^\mathbf{L} E) \oplus (G\otimes^\mathbf{L} E^\prime) 
        \cong G\otimes^\mathbf{L} (E \oplus E^\prime)
        \in \langle \mathbf{R}f_\ast D_{\operatorname{qc}}(Y) \rangle.
    \end{displaymath}
    To check the reverse inclusion, it suffices to show $\mathbf{R}f_\ast D_{\operatorname{qc}}(Y) \subseteq \langle \mathbf{R}f_\ast \mathcal{O}_Y \rangle_{\otimes}$. 
    Choose $E\in D_{\operatorname{qc}}(Y)$. 
    Then unit equations for adjunction yield a composition 
    \begin{displaymath}
        \mathbf{R}f_\ast E \to \mathbf{R}f_\ast \mathbf{L}f^\ast \mathbf{R}f_\ast E \to \mathbf{R}f_\ast E
    \end{displaymath}
    which is the identity on $\mathbf{R}f_\ast E$. 
    Hence, $\mathbf{R}f_\ast E \in \langle \mathbf{R}f_\ast \mathbf{L}f^\ast \mathbf{R}f_\ast E \rangle_1$. 
    Projection formula says that 
    \begin{displaymath}
        \mathbf{R}f_\ast \mathbf{L}f^\ast \mathbf{R}f_\ast E \cong \mathbf{R}f_\ast E \otimes^{\mathbf{L}} \mathbf{R}f_\ast \mathcal{O}_Y.
    \end{displaymath}
    Since $\mathbf{R}f_\ast E \otimes^{\mathbf{L}} \mathbf{R}f_\ast \mathcal{O}_Y\in \langle \mathbf{R}f_\ast \mathcal{O}_Y \rangle_{\otimes}$, the claim follows.
\end{proof}

\section{Descent}
\label{sec:descent}

We prove a version of descent. 

\begin{lemma}
    \label{lem:affine_descendability}
    Let $f\colon Y \to X$ be an fppf cover of quasi-compact quasi-separated algebraic spaces. 
    If $X$ is a scheme, then $\mathbf{R}f_\ast \mathcal{O}_Y$ is a descendable algebra in $D_{\operatorname{qc}}(X)$.
\end{lemma}

\begin{proof}
    Choose an \'{e}tale presentation $s\colon U \to Y$ from an affine scheme. 
    By \cite[\href{https://stacks.math.columbia.edu/tag/03Y9}{Tag 03Y9}]{StacksProject}, $f\circ s$ is an fppf cover. 
    Moreover, by \cite[Lemma 11.23(2)]{Bhatt/Scholze:2017}, $\mathbf{R}(f \circ s)_\ast \mathcal{O}_U$ is descendable. 
    Recall that there exists a natural isomorphism $\mathbf{R}f_\ast \circ \mathbf{R}s_\ast \to \mathbf{R}(f\circ s)_\ast$ \cite[\href{https://stacks.math.columbia.edu/tag/0D6E}{Tag 0D6E}]{StacksProject}. 
    Thus, from \Cref{lem:descent}, it follows that 
    \begin{displaymath}
        \mathcal{O}_X 
        \in \langle \mathbf{R}(f \circ s)_\ast D_{\operatorname{qc}}(U) \rangle \subseteq \langle \mathbf{R}f_\ast D_{\operatorname{qc}}(Y) \rangle,
    \end{displaymath}
    which yields that $\mathbf{R}f_\ast \mathcal{O}_Y$ is a descendable algebra in $D_{\operatorname{qc}}(X)$.
\end{proof}

\begin{proof}
    [Proof of \Cref{prop:fppf_descent}]
    Let $P$ be the property for quasi-compact open immersions $U\to X$ that $\mathcal{O}_U \in\langle \mathbf{R}(g_U)_\ast D_{\operatorname{qc}}(U\times_X Y) \rangle$ where $g_U \colon Y_U :=U\times_X Y \to U$ is the base change of $g$ along $U\to X$.
    We apply \cite[\href{https://stacks.math.columbia.edu/tag/09IT}{Tag 09IT}]{StacksProject}. 
    We may assume $X$ is not empty. 
    By \cite[\href{https://stacks.math.columbia.edu/tag/06NH}{Tag 06NH}]{StacksProject}, $X$ contains a nonempty subspace $X_0$ which is a scheme. 
    Choose a nonempty affine open subscheme $W_0$ of $X_0$.  
    Then $W_0$ satisfies the desired claim via \Cref{lem:affine_descendability}. 

    We check the remaining condition in \cite[\href{https://stacks.math.columbia.edu/tag/09IT}{Tag 09IT}]{StacksProject}.
    Consider an elementary distinguished square given by a quasi-compact open immersion $j\colon U \to W$ such that $W_0 \subseteq U$, an \'{e}tale morphism $f\colon X^\prime \to W$ from an affine scheme, and $W$ an open subspace of $X$. 
    Assume that $U$ satisfies the desired claim. 
    Form the fibered square 
    \begin{displaymath}
        \begin{tikzcd}
            {U^\prime} & {X^\prime} \\
            U & {W.}
            \arrow["{j^\prime}", from=1-1, to=1-2]
            \arrow["{f^\prime}"', from=1-1, to=2-1]
            \arrow["f", from=1-2, to=2-2]
            \arrow["j"', from=2-1, to=2-2]
        \end{tikzcd}
    \end{displaymath}
    By \cite[\href{https://stacks.math.columbia.edu/tag/08GW}{Tag 08GW}]{StacksProject}, there exists a distinguished triangle 
    \begin{displaymath}
        \mathcal{O}_W \to \mathbf{R}j_\ast \mathcal{O}_U \oplus \mathbf{R}f_\ast \mathcal{O}_ {X^\prime} \to \mathbf{R}(f\circ j^\prime)_\ast \mathcal{O}_{U^\prime} \to \mathcal{O}_W[1].
    \end{displaymath}
    Consider the commutative cube 
    \begin{displaymath}
        \begin{tikzcd}
            {Y^\prime_U} && {Y^\prime} & \\
            {U^\prime} && {X^\prime} \\
            & {Y_U} && {Y_W} \\
            & U && W
            \arrow["{j_Y^\prime}"{description}, from=1-1, to=1-3]
            \arrow["{g^\prime}"{description}, from=1-1, to=2-1]
            \arrow["{f^\prime_Y}"{description, pos=0.8}, from=1-1, to=3-2]
            \arrow["{g_{X^\prime}}"{description}, from=1-3, to=2-3]
            \arrow["{\widetilde{f}}"{description, pos=0.7}, from=1-3, to=3-4]
            \arrow["{j^\prime}"{description}, from=2-1, to=2-3]
            \arrow["{f^\prime}"{description, pos=0.2}, from=2-1, to=4-2]
            \arrow["f"{description, pos=0.2}, from=2-3, to=4-4]
            \arrow["{\widetilde{j}}"{description}, from=3-2, to=3-4]
            \arrow["{g_U}"{description}, from=3-2, to=4-2]
            \arrow["{g_W}"{description}, from=3-4, to=4-4]
            \arrow["j"{description}, from=4-2, to=4-4]
        \end{tikzcd}
    \end{displaymath}
    whose faces are fibered. 
    By \cite[\href{https://stacks.math.columbia.edu/tag/03Y9}{Tag 03Y9}]{StacksProject}, $g_U$, $g_{X^\prime}$, and $g^\prime$ are fppf covers.
    Since $U$ satisfies property $P$, it follows that 
    $\mathcal{O}_U \in \langle \mathbf{R}(g_U)_\ast D_{\operatorname{qc}}(Y_U)\rangle$. 
    Moreover, \Cref{lem:affine_descendability} implies $\mathcal{O}_{X^\prime} \in \langle \mathbf{R}(g_{X^\prime})_\ast D_{\operatorname{qc}}(Y^\prime)\rangle$ and $\mathcal{O}_{U^\prime} \in \langle \mathbf{R}(g^\prime)_\ast D_{\operatorname{qc}}(Y_U^\prime)\rangle$. 
    See \Cref{lem:descent}.
    Consequently, we obtain 
    \begin{displaymath}
        \begin{aligned}
            \mathcal{O}_W 
            & \in \langle \mathbf{R}j_\ast \mathcal{O}_U \oplus \mathbf{R}f_\ast \mathcal{O}_ {X^\prime} \oplus \mathbf{R}(f\circ j^\prime)_\ast \mathcal{O}_{U^\prime} \rangle
            \\&\subseteq \langle \mathbf{R}(j \circ g_U)_\ast D_{\operatorname{qc}}(Y_U) \oplus \mathbf{R}(f \circ g_{X^\prime})_\ast D_{\operatorname{qc}}(Y^\prime) \oplus \mathbf{R}(j \circ f^\prime \circ g^\prime)_\ast D_{\operatorname{qc}}(Y_U^\prime) \rangle
            \\& \subseteq \langle \mathbf{R}(g_W \circ \widetilde{j})_\ast D_{\operatorname{qc}}(Y_U) \oplus \mathbf{R}(g_W \circ \widetilde{f})_\ast D_{\operatorname{qc}}(Y^\prime) \oplus \mathbf{R}(g_W \circ \widetilde{j} \circ f^\prime_Y)_\ast D_{\operatorname{qc}}(Y_U^\prime) \rangle
            \\&\subseteq \langle \mathbf{R}(g_W)_\ast D_{\operatorname{qc}}(Y_W) \rangle.
        \end{aligned}
    \end{displaymath}
    By \Cref{lem:descent}, this completes the proof.
\end{proof}

\section{Localizing}
\label{sec:localizing}

\begin{lemma}
    \label{lem:inverse_image_localizing}
    Let $F\colon \mathcal{T} \to \mathcal{S}$ be an exact functor of triangulated categories admitting small coproducts. 
    Suppose $\mathcal{A}\subseteq \mathcal{T}$, that $F$ preserves small coproducts, and $\mathcal{B}\subseteq \mathcal{S}$ is localizing. 
    Then $F^{-1}\mathcal{B}\subseteq \mathcal{T}$ is localizing and $F(\overline{\langle \mathcal{A} \rangle}) \subseteq \overline{\langle F(\mathcal{A}) \rangle}$.
\end{lemma}

\begin{proof}
    This is well-known, but we include it for convenience. See e.g.\ \cite[Lemma 1.1]{Hall/Rydh:2017a}. 
    We prove the first claim.
    Recall that $F^{-1}\mathcal{B}\subseteq \mathcal{T}$ denotes the strictly full subcategory of $E\in \mathcal{T}$ such that $F(E)\in \mathcal{B}$. 
    If $E\in F^{-1}\mathcal{B}$ and $n\in \mathbb{Z}$, then $F(E[n])\cong F(E)[n]\in \mathcal{B}$ because $\mathcal{B}$ is closed under shifts, and so $E[n]\in F^{-1}\mathcal{B}$. 
    Moreover, for any distinguished triangle
    \begin{displaymath}
        A \to E \to B \to A[1]
    \end{displaymath}
    where $A,B\in F^{-1}\mathcal{B}$, it follows that 
    \begin{displaymath}
        F(A) \to F(E) \to F(B) \to F(A[1])
    \end{displaymath}
    is a distinguished triangle of objects in $\mathcal{B}$. 
    Hence, $E\in F^{-1} \mathcal{B}$, which shows that $F^{-1}\mathcal{B}$ is a triangulated subcategory of $\mathcal{T}$. 
    Finally, if $I$ is a set and $E_i\in F^{-1}\mathcal{B}$ for $i\in I$, then $F(\oplus_{i\in I} E_i)\cong \oplus_{i\in I} F(E_i)\in\mathcal{B}$ because $\mathcal{B}$ is closed under small coproducts and $F$ preserves small coproducts. 
    This implies that $\oplus_{i\in I} E_i\in F^{-1}\mathcal{B}$.

    We prove the second claim. 
    Since $F(\mathcal{A})\subseteq F(\overline{\langle \mathcal{A} \rangle})$, it follows that $\overline{\langle F(\mathcal{A}) \rangle} \subseteq \overline{ \langle F(\overline{\langle \mathcal{A} \rangle}) \rangle}$. 
    By the first claim, we know that $F^{-1}\overline{\langle F(\mathcal{A}) \rangle}$ is localizing. 
    As $\mathcal{A}\subseteq F^{-1}\overline{\langle F(\mathcal{A}) \rangle}$, we obtain $\overline{\langle \mathcal{A} \rangle} \subseteq F^{-1}\overline{\langle F(\mathcal{A}) \rangle}$. 
    Therefore, the desired claim follows.
\end{proof}

\begin{lemma}
    \label{lem:independent_of_representative}
    Let $X$ be a quasi-compact quasi-separated algebraic space, $S\subseteq |X|$, and fix a representative $t_s \colon \operatorname{Spec}(k)\to X$ for each $s\in S$. 
    Then
    \begin{displaymath}
        \overline{\langle \{ \mathbf{R}(t_s)_\ast \mathcal{O}_{\operatorname{Spec}(k)} \}_{s\in S} \rangle}_{\otimes} = \overline{\langle \{ \mathbf{R}(t_s)_\ast \mathcal{O}_{\operatorname{Spec}(k)} \}_{s\in S}  \rangle}.
    \end{displaymath}
    Moreover, this localizing subcategory is independent of the choice of representatives for $s\in S$.
\end{lemma}

\begin{proof}
    We prove the first claim. 
    It is easy to see that 
    \begin{displaymath}
        \overline{\langle \{ \mathbf{R}(t_s)_\ast \mathcal{O}_{\operatorname{Spec}(k)} \}_{s\in S} \rangle} \subseteq \overline{\langle \{ \mathbf{R}(t_s)_\ast \mathcal{O}_{\operatorname{Spec}(k)} \}_{s\in S}  \rangle}_{\otimes}.
    \end{displaymath}
    To prove the reverse inclusion, we show that $\overline{\langle \{ \mathbf{R}(t_s)_\ast \mathcal{O}_{\operatorname{Spec}(k)} \}_{s\in S} \rangle}$ is a $\otimes$-subcategory. 
    Let $E\in D_{\operatorname{qc}}(X)$. 
    Note that $E\otimes^{\mathbf{L}}(-)$ is a small coproduct preserving endofunctor on $D_{\operatorname{qc}}(X)$. 
    By the projection formula and \Cref{lem:inverse_image_localizing}, 
    \begin{displaymath}
        \begin{aligned}
            E \otimes^{\mathbf{L}} \overline{\langle \{ \mathbf{R}(t_s)_\ast \mathcal{O}_{\operatorname{Spec}(k)} \}_{s\in S} \rangle}
            &\subseteq \overline{\langle E \otimes^{\mathbf{L}}  \{ \mathbf{R}(t_s)_\ast \mathcal{O}_{\operatorname{Spec}(k)} \}_{s\in S} \rangle}
            \\&\subseteq \overline{\langle \{ E \otimes^{\mathbf{L}}  \mathbf{R}(t_s)_\ast \mathcal{O}_{\operatorname{Spec}(k)} \}_{s\in S} \rangle}
            \\&\subseteq \overline{\langle \{ \mathbf{R}(t_s)_\ast \mathbf{L}t_s^\ast E \}_{s\in S} \rangle}.
        \end{aligned}
    \end{displaymath}
    Moreover, for each $s\in S$, $D_{\operatorname{qc}}(\operatorname{Spec}(k)) = \overline{\langle \mathcal{O}_{\operatorname{Spec}(k)} \rangle}$. 
    By \Cref{lem:inverse_image_localizing}, 
    \begin{displaymath}
        \begin{aligned}
            \overline{\langle \mathbf{R}(t_s)_\ast \mathbf{L}t_s^\ast E \rangle}
            \subseteq \overline{\langle \mathbf{R}(t_s)_\ast \mathcal{O}_{\operatorname{Spec}(k)} \rangle}.
        \end{aligned}
    \end{displaymath}
    Consequently, it follows that 
    \begin{displaymath}
        \begin{aligned}
            E \otimes^{\mathbf{L}} \overline{\langle \{ \mathbf{R}(t_s)_\ast \mathcal{O}_{\operatorname{Spec}(k)} \}_{s\in S} \rangle}
            &\subseteq\overline{\langle \{ \mathbf{R}(t_s)_\ast \mathbf{L}t_s^\ast E \}_{s\in S} \rangle}
            \\&\subseteq \overline{\langle \{ \mathbf{R}(t_s)_\ast \mathcal{O}_{\operatorname{Spec}(k)} \}_{s\in S} \rangle}.
        \end{aligned}
    \end{displaymath}

    Next we check the second claim. 
    Choose any collection of representatives $v_s \colon \operatorname{Spec}(k^\prime) \to X$ for each $s\in S$. 
    Since $t_s$ and $v_s$ represent $s$, there exist a field $\ell$ and commutative diagram 
    \begin{displaymath}
        \begin{tikzcd}
            {\operatorname{Spec}(\ell)} & {\operatorname{Spec}(k^\prime)} \\
            {\operatorname{Spec}(k)} & {X.}
            \arrow["{t^\prime_s}", from=1-1, to=1-2]
            \arrow["{v_s^\prime}"', from=1-1, to=2-1]
            \arrow["{v_s}", from=1-2, to=2-2]
            \arrow["{t_s}"', from=2-1, to=2-2]
        \end{tikzcd}
    \end{displaymath}
    Note that $t^\prime_s$ and $v^\prime_s$ are faithfully flat. 
    Hence, the derived pushforward of $\mathcal{O}_{\operatorname{Spec}(\ell)}$ is nonzero along each morphism. 
    In particular, $\mathbf{R}(t^\prime_s)_\ast \mathcal{O}_{\operatorname{Spec}(\ell)}$ is a nonzero coproduct of shifts of $\mathcal{O}_{\operatorname{Spec}(k^\prime)}$, and similarly for $\mathbf{R} (v^\prime_s)_\ast \mathcal{O}_{\operatorname{Spec}(\ell)}$. 
    Consequently, 
    \begin{displaymath}
        \overline{\langle \mathbf{R}(t^\prime_s)_\ast \mathcal{O}_{\operatorname{Spec}(\ell)} \rangle} 
        = D_{\operatorname{qc}}(\operatorname{Spec}(k^\prime)) 
        = \overline{\langle \mathcal{O}_{\operatorname{Spec}(k^\prime)} \rangle}.
    \end{displaymath}
    A similar statement holds for $\mathbf{R} (v^\prime_s)_\ast \mathcal{O}_{\operatorname{Spec}(\ell)}$ in $D_{\operatorname{qc}}(\operatorname{Spec}(k))$. 
    By \Cref{lem:inverse_image_localizing}, we have that 
    \begin{displaymath}
        \begin{aligned}
            \mathbf{R} & (v_s)_\ast \overline{\langle \mathbf{R}(t^\prime_s)_\ast \mathcal{O}_{\operatorname{Spec}(\ell)} \rangle}
            \subseteq \overline{\langle \mathbf{R} (v_s)_\ast \mathcal{O}_{\operatorname{Spec}(k^\prime)} \rangle}, 
            \\& \mathbf{R} (v_s)_\ast \overline{\langle \mathcal{O}_{\operatorname{Spec}(k^\prime)} \rangle}
            \subseteq \overline{\langle \mathbf{R} (v_s)_\ast \mathbf{R}(t^\prime_s)_\ast \mathcal{O}_{\operatorname{Spec}(\ell)}\rangle}.
        \end{aligned}
    \end{displaymath}
    This implies that
    \begin{displaymath}
        \overline{\langle \mathbf{R} (v_s)_\ast \mathcal{O}_{\operatorname{Spec}(k^\prime)} \rangle}
        =\overline{\langle \mathbf{R} (v_s)_\ast \mathbf{R}(t^\prime_s)_\ast \mathcal{O}_{\operatorname{Spec}(\ell)}\rangle}.
    \end{displaymath}
    A similar statement holds for $\mathbf{R} (v^\prime_s)_\ast \mathcal{O}_{\operatorname{Spec}(\ell)}$. 
    Therefore, it follows that
    \begin{displaymath}
        \overline{\langle \mathbf{R} (v_s)_\ast \mathcal{O}_{\operatorname{Spec}(k^\prime)} \rangle}
        = \overline{\langle \mathbf{R} (v_s)_\ast \mathbf{R}(t^\prime_s)_\ast \mathcal{O}_{\operatorname{Spec}(\ell)}\rangle}
        = \overline{\langle \mathbf{R} (t_s)_\ast \mathcal{O}_{\operatorname{Spec}(k)} \rangle},
    \end{displaymath}
    which completes the proof.
\end{proof}

\begin{lemma}
    \label{lem:tensor_to_normal}
    Let $X$ be a quasi-compact quasi-separated algebraic space. 
    Suppose $G$ is a compact generator for $D_{\operatorname{qc}}(X)$. 
    If $\mathcal{S}\subseteq D_{\operatorname{qc}}(X)$ is a set, then $\overline{\langle \mathcal{S} \rangle}_{\otimes} = \overline{\langle G \otimes^{\mathbf{L}} \mathcal{S} \rangle}$.
\end{lemma}

\begin{proof}
    It is easy to see that $\overline{\langle G \otimes^{\mathbf{L}} \mathcal{S} \rangle} \subseteq \overline{\langle \mathcal{S} \rangle}_{\otimes}$. We prove the reverse inclusion. 
    We show that $\overline{\langle G \otimes^{\mathbf{L}} \mathcal{S} \rangle}$ is an $\otimes$-subcategory. 
    By \cite[\href{https://stacks.math.columbia.edu/tag/09SN}{Tag 09SN}]{StacksProject}, every $E\in D_{\operatorname{qc}}(X)$ admits a distinguished triangle
    \begin{displaymath}
        \oplus_{n\geq 1} P_n \to \oplus_{n\geq 1} P_n \to E \to (\oplus_{n\geq 1} P_n)[1]
    \end{displaymath}
    where $P_n \in \overline{\langle G \rangle}_n$. 
    To show that $E\otimes^{\mathbf{L}} \overline{\langle G \otimes^{\mathbf{L}} \mathcal{S} \rangle} \subseteq \overline{\langle G \otimes^{\mathbf{L}} \mathcal{S} \rangle}$, it suffices to prove that $\oplus_{n\geq 1} P_n \otimes^{\mathbf{L}} \overline{\langle G \otimes^{\mathbf{L}} \mathcal{S} \rangle} \subseteq \overline{\langle G \otimes^{\mathbf{L}} \mathcal{S} \rangle}$. 
    
    As localizing subcategories are closed under small coproducts, we can reduce to proving that each $P_n \otimes^{\mathbf{L}} \overline{\langle G \otimes^{\mathbf{L}} \mathcal{S} \rangle} \subseteq \overline{\langle G \otimes^{\mathbf{L}} \mathcal{S} \rangle}$. 
    However, $P_n \in \overline{\langle G \rangle}_n$, and so it suffices to verify that $\overline{\langle G \rangle}_n \otimes^{\mathbf{L}} \overline{\langle G \otimes^{\mathbf{L}} \mathcal{S} \rangle} \subseteq \overline{\langle G \otimes^{\mathbf{L}} \mathcal{S} \rangle}$. 
    The claim now follows by induction on $n$. 
    Indeed, let $A\in \overline{\langle G \rangle}_1$. 
    Then $A$ is a direct summand of a small coproduct of shifts of $G$, say $\oplus_{i\in I} G[i]$. 
    So, for any $S\in \mathcal{S}$,
    \begin{displaymath}
        (\bigoplus_{i\in I} G[i]) \otimes^{\mathbf{L}} (G \otimes^{\mathbf{L}} S) \cong \bigoplus_{i\in I} (G[i] \otimes^{\mathbf{L}} G \otimes^{\mathbf{L}} S).
    \end{displaymath}
    Note that $G[i] \otimes^{\mathbf{L}} G\in \langle G \rangle$.
    Indeed, $G[i] \otimes^{\mathbf{L}} G$ is compact and $G$ is a compact generator, so it is a classifical generator for $\operatorname{Perf}(X)$ \cite[\href{https://stacks.math.columbia.edu/tag/09M8}{Tags 09M8}, \href{https://stacks.math.columbia.edu/tag/09IU}{09IU}, \href{https://stacks.math.columbia.edu/tag/09SR}{09SR}]{StacksProject}.
    Then it follows that each $G[i] \otimes^{\mathbf{L}} G \otimes^{\mathbf{L}} S\in \overline{\langle G \otimes^{\mathbf{L}} \mathcal{S} \rangle}$. 
    By \Cref{lem:inverse_image_localizing}, 
    \begin{displaymath}
        (\bigoplus_{i\in I} G[i]) \otimes^{\mathbf{L}} \overline{\langle G \otimes^{\mathbf{L}} \mathcal{S} \rangle} \subseteq \overline{\langle (\bigoplus_{i\in I} G[i]) \otimes^{\mathbf{L}} G \otimes^{\mathbf{L}} \mathcal{S} \rangle} \subseteq \overline{\langle G \otimes^{\mathbf{L}} \mathcal{S} \rangle}.
    \end{displaymath} 

    Now, let $E\in \overline{\langle G \rangle}_{n+1}$. 
    There exists a distinguished triangle
    \begin{displaymath}
        A \to E \oplus E^\prime \to B \to A[1]
    \end{displaymath}
    where $A\in \overline{\langle G \rangle}_n$ and $B\in \overline{\langle G \rangle}_1$. 
    Choose any $Q\in \overline{\langle G \otimes^{\mathbf{L}} \mathcal{S} \rangle}$. 
    Tensoring with $Q$ gives a distinguished triangle
    \begin{displaymath}
        A \otimes^{\mathbf{L}} Q \to E \otimes^{\mathbf{L}} Q  \oplus E^\prime  \otimes^{\mathbf{L}} Q  \to B \otimes^{\mathbf{L}} Q \to A[1] \otimes^{\mathbf{L}} Q 
    \end{displaymath}
    By the induction hypothesis, it follows that $A \otimes^{\mathbf{L}} Q,B \otimes^{\mathbf{L}} Q \in \overline{\langle G \otimes^{\mathbf{L}} \mathcal{S} \rangle}$. 
    Hence, $E \otimes^{\mathbf{L}} Q \in \overline{\langle G \otimes^{\mathbf{L}} \mathcal{S} \rangle}$.

    Since $\mathcal{O}_X \in \langle G \rangle$, it follows that for each $E\in \mathcal{S}$,
    \begin{displaymath}
        E \in \langle G \rangle \otimes^{\mathbf{L}} E \subseteq  \langle G \otimes^{\mathbf{L}} E \rangle \subseteq \overline{\langle G \otimes^{\mathbf{L}} E \rangle}.
    \end{displaymath}
    Thus, $\overline{\langle G \otimes^{\mathbf{L}} \mathcal{S} \rangle}$ is a $\otimes$-subcategory which contains $\mathcal{S}$.
\end{proof}

\begin{lemma}
    \label{lem:subcategory_to_object_generation}
    Let $\mathcal{T}$ be a triangulated category and $\mathcal{S}\subseteq\mathcal{T}$ be a subcategory. 
    For each $E\in\langle \mathcal{S} \rangle_n$, there exists an $S\in\langle \mathcal{S} \rangle_1$ with $E\in \langle S \rangle_n$.
\end{lemma}

\begin{proof}
    This is essentially \cite[Lemma 5.5]{DeDeyn/Lank/ManaliRahul:2025} but we rephrase it a bit. 
    There is nothing to prove when $E=0$, and so, we can assume that $E\not\cong 0$. 
    In this case, $n>0$. 
    If $n=1$, then $S=E$ satisfies the desired claim. 
    Assume the desired claim holds for all $A\in \langle \mathcal{S} \rangle_c$ with $0\leq c \leq n$. 
    Let $E\in \langle\mathcal{S}\rangle_{n+1}$. This gives us a distinguished triangle 
    \begin{displaymath}
        A \to E \oplus E^\prime \to B \to A[1]
    \end{displaymath}
    with $A\in\langle \mathcal{S} \rangle_n$ and $B\in\langle \mathcal{S} \rangle_1$. 
    By the induction hypothesis, there exists $S^\prime\in \langle\mathcal{S}\rangle_1$ with $A \in \langle S^\prime \rangle_n$. Define $S := S^\prime \oplus B$. 
    Clearly, $S\in\langle \mathcal{S} \rangle_1$, and so the distinguished triangle above shows $E\in\langle S \rangle_{n+1}$. 
    This completes the proof.
\end{proof}

\section{fppf covers}
\label{sec:fppf covers}

\begin{lemma}
    \label{lem:localizing_is_generated_by_fields}
    Let $X$ be a quasi-compact quasi-separated algebraic space. 
    Assume there exists an fppf cover $s\colon U \to X$ from a point-generated affine scheme. 
    If $\mathcal{T}$ is a $\otimes$-localizing subcategory of $D_{\operatorname{qc}}(X)$, then there exists $S\subseteq |X|$ such that $\mathcal{T} = \overline{\langle \{ \mathbf{R}(t_s)_\ast \mathcal{O}_{\operatorname{Spec}(k)} \}_{s\in S}  \rangle}$ where $t_s \colon \operatorname{Spec}(k)\to X$ is a representative for each $s\in S$. 
    In particular, $X$ is point-generated.
\end{lemma}

\begin{proof}
    By \Cref{prop:fppf_descent}, there exists an $n\geq 0$ such that $\mathcal{O}_{X} \in \langle \mathbf{R}s_\ast D_{\operatorname{qc}} (U) \rangle_n$. 
    By \Cref{lem:subcategory_to_object_generation}, there exists $E\in D_{\operatorname{qc}}(U)$ such that $\mathcal{O}_{X} \in \langle \mathbf{R}s_\ast E \rangle_n$. 
    Indeed, there exists $E^\prime \in \langle \mathbf{R}s_\ast D_{\operatorname{qc}} (U) \rangle_1$ such that $\mathcal{O}_{X} \in \langle E^\prime \rangle_n$.
    However, $E^\prime$ is a direct summand of a finite coproduct of the form $\oplus_{i\in I} \mathbf{R}s_\ast E_i$.
    Taking $E:= \oplus_{i\in I} E_i$ and using that $\mathbf{R}s_\ast$ commutes with small coproducts shows that $\mathcal{O}_{X} \in \langle \mathbf{R}s_\ast E \rangle_n$. 
    
    Fix a $\otimes$-localizing subcategory $\mathcal{T}\subseteq D_{\operatorname{qc}} (X)$. 
    By \Cref{lem:inverse_image_localizing}, $s^{-1}\mathcal{T}$ is localizing in $D_{\operatorname{qc}}(U)$. 
    Since $U$ is affine, $s^{-1} \mathcal{T}$ is $\otimes$-localizing because every localizing subcategory is $\otimes$-closed (e.g.\ use that $\mathcal{O}_U$ is a compact generator). 
    Denote by $h_q \colon \operatorname{Spec}(\kappa(q))\to U$ the natural morphisms for each $q\in |U|$. 
    By \cite[\href{https://stacks.math.columbia.edu/tag/03BV}{Tags 03BV} \& \href{https://stacks.math.columbia.edu/tag/03BU}{03BU}]{StacksProject}, the points of $U$ (viewed as a topological space) are in one-to-one correspondence with those of $|U|$. 
    By \cite[Theorem 0.1]{Alonso/Jeremias/Loureiro:2024} and \cite[\href{https://stacks.math.columbia.edu/tag/071Q}{Tag 071Q}]{StacksProject}, there exists $T\subseteq |U|$ such that 
    \begin{displaymath}
        \overline{\langle \{ \mathbf{R} (h_q)_\ast \mathcal{O}_{\operatorname{Spec}(\kappa(q))} \}_{q\in T} \rangle} = s^{-1} \mathcal{T}.
    \end{displaymath}
    
    Note that the construction of $s^{-1} \mathcal{T}$ implies 
    \begin{displaymath}
        \overline{\langle \{ \mathbf{R} (s\circ h_q)_\ast \mathcal{O}_{\operatorname{Spec}(\kappa(q))} \}_{q\in T} \rangle} \subseteq \mathcal{T}.
    \end{displaymath}
    We prove the reverse inclusion. 
    Let $A\in \mathcal{T}$. 
    By tensoring with $A$, we obtain $A \in \langle \mathbf{R}s_\ast E \otimes^{\mathbf{L}} A \rangle_n$. 
    By the projection formula, $\mathbf{R}s_\ast E \otimes^{\mathbf{L}} A  \cong \mathbf{R}s_\ast (E \otimes^{\mathbf{L}} \mathbf{L} s^\ast A)$. 
    Since $A\in \mathcal{T}$ and $\mathcal{T}$ is $\otimes$-localizing, we know that $\mathbf{R}s_\ast E \otimes^{\mathbf{L}} A \in \mathcal{T}$. 
    Hence, $E \otimes^{\mathbf{L}} \mathbf{L} s^\ast A\in s^{-1} \mathcal{T}$, which implies 
    \begin{displaymath}
        \mathbf{R}s_\ast (E \otimes^{\mathbf{L}} \mathbf{L} s^\ast A) \in \mathbf{R}s_\ast s^{-1} \mathcal{T}.
    \end{displaymath}
    Applying \Cref{lem:inverse_image_localizing},
    \begin{displaymath}
        \mathbf{R}s_\ast s^{-1} \mathcal{T} = \mathbf{R}s_\ast \overline{\langle \{ \mathbf{R} (h_q)_\ast \mathcal{O}_{\operatorname{Spec}(\kappa(q))} \}_{q\in T} \rangle} \subseteq \overline{\langle \{ \mathbf{R} (s\circ h_q)_\ast \mathcal{O}_{\operatorname{Spec}(\kappa(q))} \}_{q\in T} \rangle}.
    \end{displaymath} 
    Thus,
    \begin{displaymath}
        A \in \langle \mathbf{R}s_\ast E \otimes^{\mathbf{L}} A \rangle_n \subseteq \overline{\langle \{ \mathbf{R} (s\circ h_q)_\ast \mathcal{O}_{\operatorname{Spec}(\kappa(q))} \}_{q\in T} \rangle}.
    \end{displaymath}
    By \Cref{lem:independent_of_representative}, the choice of representatives for $s(T)\subseteq |X|$ does not matter.
\end{proof}

\begin{lemma}
    \label{lem:ascent_pt_generated}
    Let $f\colon Y\to X$ be an fppf cover of quasi-compact quasi-separated algebraic spaces. 
    If $X$ is point-generated, then $Y$ is point-generated.
\end{lemma}

\begin{proof}
    Choose an \'{e}tale presentation $s\colon U \to Y$ with $U$ affine. 
    By \Cref{lem:localizing_is_generated_by_fields}, it suffices to show that $U$ is point-generated. 
    By \Cref{prop:fppf_descent} and \Cref{lem:descent}, $\mathcal{O}_X \in \langle \mathbf{R}(f\circ s)_\ast D_{\operatorname{qc}}(U) \rangle$. 
    In fact, there exists $n\geq 0$ such that $\mathcal{O}_X \in \langle \mathbf{R}(f\circ s)_\ast D_{\operatorname{qc}}(U) \rangle_n$.
    Applying $\mathbf{L}(f\circ s)^\ast$, it follows that 
    \begin{displaymath}
        \mathcal{O}_U \in \langle \mathbf{L}(f\circ s)^\ast \mathbf{R}(f\circ s)_\ast D_{\operatorname{qc}}(U) \rangle_n \subseteq \langle \mathbf{L}(f\circ s)^\ast D_{\operatorname{qc}}(X) \rangle_n.
    \end{displaymath}
    Since $X$ is point-generated, there exists $S\subseteq |X|$ such that $D_{\operatorname{qc}}(X) =\overline{\langle \big\{ \mathbf{R}(t_s)_\ast \mathcal{O}_{\operatorname{Spec}(k)} \big\}_{s\in S} \rangle}$ where $t_s$ is some representative of $s\in S$. 
    As $\mathbf{L}(f\circ s)^\ast$ preserves small coproducts, \Cref{lem:inverse_image_localizing} implies 
    \begin{displaymath}
        \mathbf{L}(f\circ s)^\ast \overline{\langle \big\{ \mathbf{R}(t_s)_\ast \mathcal{O}_{\operatorname{Spec}(k)} \big\}_{s\in S} \rangle} 
        \subseteq \overline{\langle \big\{ \mathbf{L}(f\circ s)^\ast  \mathbf{R}(t_s)_\ast \mathcal{O}_{\operatorname{Spec}(k)} \big\}_{s\in S} \rangle}
    \end{displaymath}
    Now, a localizing subcategory is triangulated, so it is closed under extensions and shifts, which implies 
    \begin{displaymath}
        \mathcal{O}_U 
        \in \langle \mathbf{L}(f\circ s)^\ast D_{\operatorname{qc}}(X) \rangle_n
        \subseteq \overline{\langle \big\{ \mathbf{L}(f\circ s)^\ast  \mathbf{R}(t_s)_\ast \mathcal{O}_{\operatorname{Spec}(k)} \big\}_{s\in S} \rangle}.
    \end{displaymath}
    Fix $t_s$. 
    Consider the fibered square 
    \begin{displaymath}
        \begin{tikzcd}
            {U\times_X \operatorname{Spec}(k)} & {\operatorname{Spec}(k)} \\
            U & {X.}
            \arrow["{p_2^s}", from=1-1, to=1-2]
            \arrow["{p_1^s}"', from=1-1, to=2-1]
            \arrow["{t_s}", from=1-2, to=2-2]
            \arrow["{f\circ s}"', from=2-1, to=2-2]
        \end{tikzcd}
    \end{displaymath}
    Base change says $p_2^s$ is an fppf cover. 
    By \cite[\href{https://stacks.math.columbia.edu/tag/09TF}{Tag 09TF}]{StacksProject}, $t_s$ is affine.
    Again, by base change, $p_1^s$ is affine.
    Hence, $U\times_X \operatorname{Spec}(k)$ is affine, and so $U\times_X \operatorname{Spec}(k)$ is finitely presented over $k$.  
    This ensures that $U\times_X \operatorname{Spec}(k)$ is Noetherian. 
    By \cite[Theorem 2.8]{Neeman:1992b} and \cite[\href{https://stacks.math.columbia.edu/tag/071Q}{Tag 071Q}]{StacksProject}, $U\times_{X} \operatorname{Spec}(k)$ is point-generated, and so,
    \begin{displaymath}
        \overline{\langle \mathcal{O}_{U\times_{X} \operatorname{Spec}(k)} \rangle} 
        = \overline{\langle \big\{ \mathbf{R}(v_l)_\ast \mathcal{O}_{\operatorname{Spec}(\kappa(l))} \big\}_{l\in |{U\times_{X} \operatorname{Spec}(k)|}} \rangle}.
    \end{displaymath}
    Flat base change implies 
    \begin{displaymath}
        \mathbf{R}(p_1^s)_\ast \mathcal{O}_{U\times_{X} \operatorname{Spec}(k)} \cong \mathbf{L}(f\circ s)^\ast  \mathbf{R}(t_s)_\ast \mathcal{O}_{\operatorname{Spec}(k)}.
    \end{displaymath}
    Since $\mathbf{R}(p^s_1)_\ast$ admits a right adjoint \cite[\href{https://stacks.math.columbia.edu/tag/0E55}{Tag 0E55}]{StacksProject} it preserves small coproducts. 
    Thus, \Cref{lem:inverse_image_localizing} implies 
    \begin{displaymath}
        \mathbf{R}(p_1^s)_\ast \overline{\langle \mathcal{O}_{U\times_{X} \operatorname{Spec}(k)} \rangle} 
        \subseteq \overline{\langle \big\{ \mathbf{R}(p_1^s\circ v_l)_\ast \mathcal{O}_{\operatorname{Spec}(\kappa(l))} \big\}_{l\in |{U\times_{X} \operatorname{Spec}(k)|}} \rangle}.
    \end{displaymath}
    Consequently, we obtain that 
    \begin{displaymath}
        \begin{aligned}
            \mathcal{O}_U 
            &\in \langle \mathbf{L}(f\circ s)^\ast D_{\operatorname{qc}}(X) \rangle_n
            \\&\subseteq \overline{\langle \big\{ \mathbf{L}(f\circ s)^\ast  \mathbf{R}(t_s)_\ast \mathcal{O}_{\operatorname{Spec}(k)} \big\}_{s\in S} \rangle}
            \\&\subseteq \overline{\langle \bigg\{ \big\{ \mathbf{R}(p_1^s\circ v_l)_\ast \mathcal{O}_{\operatorname{Spec}(\kappa(l))} \big\}_{l\in |{U\times_{X} \operatorname{Spec}(k)|}} \bigg\}_{s\in S} \rangle}.
        \end{aligned}
    \end{displaymath}
    However, \Cref{lem:independent_of_representative} says the latter category above is closed under tensoring by objects from $D_{\operatorname{qc}}(U)$.
    Therefore, since this latter category contains the tensor unit $\mathcal{O}_U$, it follows that $U$ is point-generated.
\end{proof}

\begin{proof}
    [Proof of \Cref{prop:ascent_descent_pt_generated}]
    This follows from \Cref{lem:localizing_is_generated_by_fields,lem:ascent_pt_generated}.
\end{proof}

\begin{proof}
    [Proof of \Cref{thm:point_generated_classification}]
    Choose an \'{e}tale presentation $s\colon U\to X$ with $U$ affine. 
    In particular, $s$ is an fppf cover.
    Since $X$ is point-generated, \Cref{prop:ascent_descent_pt_generated} implies that $U$ is point-generated.
    If $u\in |U|$, let $h_u\colon \operatorname{Spec} (\kappa(u))\to U$ be the natural morphism.

    \Cref{lem:independent_of_representative} implies that whether $\mathbf{R}t_\ast \mathcal{O}_{\operatorname{Spec}(k)}$ belongs to a
    $\otimes$-localizing subcategory depends only on the point represented
    by $t\colon \operatorname{Spec}(k) \to X$. 
    The same result shows that $\theta(S)$ is independent of the
    chosen representatives and is $\otimes$-localizing. 
    Thus, both
    $\phi$ and $\theta$ are well-defined.

    We first prove that $\theta\circ\phi$ is the identity. 
    Let
    $\mathcal{T}\subseteq D_{\operatorname{qc}}(X)$ be a $\otimes$-localizing
    subcategory. 
    By \Cref{lem:localizing_is_generated_by_fields}, there exists a subset $S\subseteq |X|$ such that $\mathcal{T}=\theta(S)$.
    Since every point of $S$ is represented by $t\colon \operatorname{Spec}(k)\to X$ for which $\mathbf{R}t_\ast \mathcal{O}_{\operatorname{Spec}(k)}$ belonging to
    $\mathcal{T}$, we have $S\subseteq\phi(\mathcal{T})$. 
    Consequently, $\mathcal{T}=\theta(S)\subseteq\theta(\phi(\mathcal{T}))$.

    Conversely, let $p\in\phi(\mathcal{T})$. 
    By the definition of $\phi$,
    there is a representative $t\colon\operatorname{Spec} (k)\to X$ of $p$ such that $\mathbf{R}t_\ast \mathcal{O}_{\operatorname{Spec}(k)}\in\mathcal{T}$. 
    By \Cref{lem:independent_of_representative}, derived pushforward of a structure sheaf associated to any other representative of $p$ also
    belongs to $\mathcal{T}$. 
    It follows that every generator of
    $\theta(\phi(\mathcal{T}))$ belongs to $\mathcal{T}$, and hence $\theta(\phi(\mathcal{T}))\subseteq\mathcal{T}$. 
    Therefore, $\theta\circ\phi=1$. 

    We next prove that $\phi\circ\theta$ is the identity. Fix a subset
    $S\subseteq |X|$. For every $p\in S$, choose a point $r_p\in |U|$ such that $|s|(r_p)=p$. 
    Set $t_p:=s\circ h_{r_p}\colon \operatorname{Spec} (\kappa(r_p))\to X$. 
    The morphism $t_p$ represents $p$ \cite[Lemma 2.4]{GuisadoVillaalgordo/Lank:2026}, so \Cref{lem:independent_of_representative} allows us to compute $\theta(S)$ using these representatives:
    \begin{displaymath}
        \theta(S) = \overline{\langle \{ \mathbf{R}(t_p)_\ast \mathcal{O}_{\operatorname{Spec} (\kappa(r_p))} \}_{p\in S} \rangle}.
    \end{displaymath}
    It is immediate that $S\subseteq\phi(\theta(S))$. 
    We are left to prove the reverse inclusion.

    Let $q\in\phi(\theta(S))$.
    Since $s$ is surjective, choose $r\in |U|$ such that $|s|(r)=q$.
    Put $t_r:=s\circ h_r\colon \operatorname{Spec} (\kappa(r))\to X$. 
    Since $t_r$ represents $q$, the well-definedness of $\phi$ gives
    $\mathbf{R}(t_r)_\ast \mathcal{O}_{\operatorname{Spec} (\kappa(r))}\in\theta(S)$.
    Applying $\mathbf{L}s^\ast $ and \Cref{lem:inverse_image_localizing}, we obtain
    \begin{equation}
        \label{eq:pullback-point-r}
        \mathbf{L}s^\ast \mathcal \mathbf{R}(t_r)_\ast \mathcal{O}_{\operatorname{Spec} (\kappa(r))}
        \in
        \overline{\langle \{ \mathbf{L}s^\ast \mathcal \mathbf{R}(t_p)_\ast \mathcal{O}_{\operatorname{Spec} (\kappa(r_p))}\}_{p\in S}
        \rangle}.
    \end{equation}

    For any $u\in |U|$, form the fibered square
    \begin{displaymath}
        \begin{tikzcd}
            W_u:=U\times_X\operatorname{Spec} (\kappa(u))
                \arrow[r,"b_u"]
                \arrow[d,"a_u"']
            &
            \operatorname{Spec} (\kappa(u))
                \arrow[d,"t_u:=s\circ h_u"]
            \\
            U
                \arrow[r,"s"']
            &
            X.
        \end{tikzcd}
    \end{displaymath}
    There is a canonical morphism $e_u\colon\operatorname{Spec} (\kappa(u))\to W_u$ induced by the pair $(h_u,1)$, and it satisfies $a_u\circ e_u=h_u$ and $b_u\circ e_u=1_{\operatorname{Spec} (\kappa(u))}$.

    By \cite[\href{https://stacks.math.columbia.edu/tag/09TF}{Tag 09TF}]{StacksProject}, the morphism $t_u$ is affine.
    Therefore, its base change $a_u$ is affine. Since $U$ is affine, it
    follows that $W_u$ is affine. 
    Moreover, $b_u$ is the base change of the
    quasi-compact \'{e}tale morphism $s$, so $b_u$ is finitely presented. 
    Thus,
    $W_u$ is an affine Noetherian scheme.
    Flat base change gives
    \begin{displaymath}
        \mathbf{L}s^\ast \mathcal \mathbf{R}(t_u)_\ast \mathcal{O}_{\operatorname{Spec} (\kappa(u))} \cong \mathbf{R}(a_u)_\ast \mathcal{O}_{W_u}.
    \end{displaymath}
    In particular, \eqref{eq:pullback-point-r} becomes
    \begin{equation}
        \label{eq:fiber-r-contained}
        \mathbf{R}(a_r)_\ast \mathcal{O}_{W_r}
        \in
        \overline{\langle \{
        \mathbf{L}s^\ast \mathcal \mathbf{R}(t_p)_\ast \mathcal{O}_{\operatorname{Spec} (\kappa(r_p))} \}_{p\in S} \rangle}.
    \end{equation}

    Since $W_r$ is affine, $D_{\operatorname{qc}}(W_r) = \overline{\langle \mathcal{O}_{W_r} \rangle}$, and so $\mathbf{R}(e_r)_\ast \mathcal{O}_{\operatorname{Spec} (\kappa(r))}\in \overline{\langle \mathcal{O}_{W_r} \rangle}$. 
    Applying $\mathbf{R}(a_r)_\ast $ and \Cref{lem:inverse_image_localizing} gives
    \begin{displaymath}
        \begin{aligned}
            \mathbf{R}(h_r)_\ast \mathcal{O}_{\operatorname{Spec} (\kappa(r))}
            &\cong
            \mathbf{R}(a_r)_\ast \mathbf{R}(e_r)_\ast \mathcal{O}_{\operatorname{Spec} (\kappa(r))}
            \\
            &\in
            \overline{\langle
            \mathbf{R}(a_r)_\ast \mathcal{O}_{W_r}
            \rangle}.
        \end{aligned}
    \end{displaymath}
    Combining this with \eqref{eq:fiber-r-contained}, we obtain
    \begin{equation}
        \label{eq:point-r-in-pullbacks}
            \mathbf{R}(h_r)_\ast \mathcal{O}_{\operatorname{Spec} (\kappa(r))}
            \in
            \overline{\langle
            \{ \mathbf{L}s^\ast \mathbf{R}(t_p)_\ast \mathcal{O}_{\operatorname{Spec} (\kappa(r_p))}\}_{p\in S}
            \rangle}.
    \end{equation}

    We now describe the localizing subcategory on the right. For each
    $p\in S$, abbreviate $W_p:=W_{r_p}$, $a_p:=a_{r_p}$, and $b_p:=b_{r_p}$.
    Flat base change gives
    \begin{equation}
        \label{eq:pullback-point-p}
        \mathbf{L}s^\ast \mathbf{R}(t_p)_\ast \mathcal{O}_{\operatorname{Spec} (\kappa(r_p))}
        \cong
        \mathbf{R}(a_p)_\ast \mathcal{O}_{W_p}.
    \end{equation}

    For every $l\in |W_p|$, let $v_l\colon\operatorname{Spec} (\kappa(l))\to W_p$ denote the natural morphism. 
    Since $W_p$ is an affine Noetherian
    scheme, \cite[Theorem 2.8]{Neeman:1992b} and \cite[\href{https://stacks.math.columbia.edu/tag/071Q}{Tag 071Q}]{StacksProject} give
    \begin{displaymath}
        D_{\operatorname{qc}}(W_p) =
        \overline{\langle \{ \mathbf{R}(v_l)_\ast \mathcal{O}_{\operatorname{Spec} (\kappa(l))}\}_{l\in |W_p|} \rangle}.
    \end{displaymath}
    On the other hand, because $W_p$ is affine, $D_{\operatorname{qc}}(W_p) =
    \overline{\langle \mathcal{O}_{W_p} \rangle}$. 
    Applying $\mathbf{R}(a_p)_\ast $ and \Cref{lem:inverse_image_localizing} in both directions therefore yields
    \begin{equation}
        \label{eq:fiber-point-generators}
        \overline{\langle \mathbf{R}(a_p)_\ast \mathcal{O}_{W_p} \rangle}
        = \overline{\langle \{ \mathbf{R}(a_p\circ v_l)_\ast 
        \mathcal{O}_{\operatorname{Spec} (\kappa(l))} \}_{\in |W_p|} \rangle }.
    \end{equation}

    Define $\Sigma_p := |a_p(|W_p|) \subseteq |U|$ and $\Sigma := \cup_{p\in S}\Sigma_p \subseteq |U|$.
    For every $l\in |W_p|$, the morphism $a_p\circ v_l\colon
    \operatorname{Spec} (\kappa(l))\to U$ represents the point $|a_p|(l)\in\Sigma_p$. 
    Hence, applying \Cref{lem:independent_of_representative}, 
    \eqref{eq:pullback-point-p} and
    \eqref{eq:fiber-point-generators} imply
    \begin{equation}
        \label{eq:pullbacks-correspond-to-sigma}
        \overline{\langle
        \{ \mathbf{L}s^\ast \mathcal \mathbf{R}(t_p)_\ast \mathcal{O}_{\operatorname{Spec} (\kappa(r_p))} \}_{p\in S}
        \rangle}
        =
        \overline{\langle
        \{ \mathbf{R}(h_u)_\ast \mathcal{O}_{\operatorname{Spec}(\kappa(u))}
        \}_{u\in\Sigma}
        \rangle}.
    \end{equation}

    Notice also that $s\circ a_p=t_p\circ b_p$. 
    Since $t_p$ represents $p$, it follows that every point in
    $\Sigma_p$ maps to $p$. 
    Thus,
    \begin{equation}
        \label{eq:sigma-contained}
        \Sigma_p\subseteq |s|^{-1}(\{p\})
        \qquad\text{and hence}\qquad
        \Sigma\subseteq |s|^{-1}(S).
    \end{equation}
    Combining \eqref{eq:point-r-in-pullbacks} and
    \eqref{eq:pullbacks-correspond-to-sigma}, we obtain
    \begin{displaymath}
        \mathbf{R}(h_r)_\ast \mathcal{O}_{\operatorname{Spec} (\kappa(r))}
        \in
        \overline{\langle
        \{ \mathbf{R}(h_u)_\ast \mathcal{O}_{\operatorname{Spec} (\kappa(u))}
        \}_{u\in\Sigma}
    \rangle}.
    \end{displaymath}
    The scheme $U$ is point-generated. 
    Therefore,
    \cite[Theorem 0.1]{Alonso/Jeremias/Loureiro:2024} identifies its $\otimes$-localizing
    subcategories with subsets of $|U|$. 
    The preceding membership implies
    that $r\in\Sigma$. 
    By \eqref{eq:sigma-contained}, $r\in |s|^{-1}(S)$.
    Since $|s|(r)=q$, it follows that $q\in S$.
    Therefore, $\phi(\theta(S))\subseteq S$. 
    Together with the opposite inclusion proved above, this gives
    $\phi\circ\theta=1$. 
    We have now shown that $\phi$ and $\theta$ are mutually inverse.
\end{proof}

\section{Bousfield localization}
\label{sec:Bousfield}

We record several consequences of the classification that may be useful in future work. 
While not necessary for our main results, we use constructions involving $t$-structures. 
The reader is referred to \cite{Beilinson/Berstein/Deligne/Gabber:2018} for background on $t$-structures, and \cite{Hrbek/Lank/Pizzirani:2025} for details on a classification of suitable $t$-structures on algebraic spaces.
Throughout this section, let $X$ be a point-generated algebraic space. 
We recall some notions.
For $E\in D_{\operatorname{qc}}(X)$, set 
\begin{displaymath}
    \operatorname{Supp}(E) := \bigcup_{i\in \mathbb{Z}} \operatorname{Supp}(\mathcal{H}^i (E)).
\end{displaymath}
Given $S\subseteq |X|$, let 
\begin{displaymath}
    D_{\operatorname{qc},S}(X):= \{ E  \in D_{\operatorname{qc}}(X) :  \operatorname{Supp}(E)\subseteq S \}.
\end{displaymath}

\begin{lemma}
    \label{lem:supported_Bousfield}
    For any subset $S\subseteq |X|$, there exists a subset $S^\prime\subseteq S$ such that 
    \begin{displaymath}
        D_{\operatorname{qc},S}(X) = D_{\operatorname{qc},S}(X) = \overline{\langle \{ \mathbf{R}(t_s)_\ast \mathcal{O}_{\operatorname{Spec}(k_s)} \}_{s\in S^\prime} \rangle}
    \end{displaymath}
    where $t_s$ is a representative of $s\in S^\prime$. 
    If $S$ is specialization closed, then we can take $S^\prime = S$.
\end{lemma}

\begin{proof}
    Choose an \'{e}tale presentation $s\colon U \to X$ by an affine scheme. 
    First we check that $D_{\operatorname{qc},S}(X)$ is $\otimes$-localizing. 
    Let $E\in D_{\operatorname{qc},S}(X)$. 
    By \cite[\href{https://stacks.math.columbia.edu/tag/07TY}{Tag 07TY}]{StacksProject} and flatness of $s$, 
    \begin{displaymath}
        \operatorname{Supp}(\mathcal{H}^i (E)) 
        = |s|(\operatorname{Supp}(\mathcal{H}^i (\mathbf{L} s^\ast E))) 
        = |s|(\operatorname{Supp}(s^\ast \mathcal{H}^i (E))).
    \end{displaymath}
    Fix $p\not\in S$. 
    There exists $u \in |U| \setminus |s|^{-1}(S)$ such that $|s|(u)=p$.
    Note that 
    \begin{displaymath}
        u \in |U| \setminus \cup_{i \in \mathbb{Z}} s(\operatorname{Supp}(s^\ast \mathcal{H}^i (E))). 
    \end{displaymath}

    For any $A\in D_{\operatorname{qc},S}(X)$, the derived pullback of $\mathbf{L}s^\ast A$ along $\operatorname{Spec}(\mathcal{O}_{U,u})$ vanishes. 
    Hence, the derived pullback of any shift of $\mathbf{L}s^\ast A$ along $\operatorname{Spec}(\mathcal{O}_{U,u}) \to U$ vanishes. 
    Thus, $D_{\operatorname{qc},S}(X)$ is closed under shifts. 

    Next let $f\colon A \to B$ be a morphism with $A,B\in D_{\operatorname{qc},S}(X)$. 
    Then the derived pullback of $\mathbf{L}s^\ast A$ and $\mathbf{L}s^\ast B$ along $\operatorname{Spec}(\mathcal{O}_{U,u}) \to U$ vanishes. 
    Hence, the derived pullback of $\operatorname{cone}(f)$ along $\operatorname{Spec}(\mathcal{O}_{U,u}) \to U$ vanishes.
    This shows that $D_{\operatorname{qc},S}(X)$ is a triangulated subcategory of $D_{\operatorname{qc}}(X)$. 

    Let $I$ be a set. 
    Choose $E_i \in D_{\operatorname{qc},S}(X)$. 
    Consider the small coproduct $\oplus_{i\in I} E_i$ in $D_{\operatorname{qc}}(X)$. 
    Recall that derived pullbacks preserve small coproducts. 
    Hence, $\mathbf{L}s^\ast (\oplus_{i\in I} E_i) \cong \oplus_{i\in I} \mathbf{L}s^\ast E_i$.
    Moreover, by this fact, the derived pullback of $\oplus_{i\in I} \mathbf{L}s^\ast E_i$ along $\operatorname{Spec}(\mathcal{O}_{U,u}) \to U$ vanishes.
    Hence, $D_{\operatorname{qc},S}(X)$ is localizing. 

    Lastly we check that $D_{\operatorname{qc},S}(X)$ is closed under tensor by objects from $D_{\operatorname{qc}}(X)$. 
    Choose $E\in D_{\operatorname{qc}}(X)$ and $A \in D_{\operatorname{qc},S}(X)$. 
    Recall that derived pullbacks are monoidal with respect to the tensor product. 
    Hence, $\mathbf{L}s^\ast (E\otimes^{\mathbf{L}} A)\cong \mathbf{L}s^\ast  E\otimes^{\mathbf{L}}  \mathbf{L}s^\ast A$. 
    Again, by this fact, the derived pullback of $\mathbf{L}s^\ast  E\otimes^{\mathbf{L}}  \mathbf{L}s^\ast A$ along $\operatorname{Spec}(\mathcal{O}_{U,u}) \to U$ vanishes.
    This shows $D_{\operatorname{qc},S}(X)$ is an $\otimes$-localizing subcategory. 

    By \Cref{thm:point_generated_classification}, there exists a subset $S^\prime \subseteq |X|$ such that 
    \begin{displaymath}
        D_{\operatorname{qc},S}(X) = \overline{\langle \{ \mathbf{R}(t_x)_\ast \mathcal{O}_{\operatorname{Spec}(k_x)} \}_{x\in S^\prime} \rangle}
    \end{displaymath}
    where $t_x$ is a representative of $x\in S^\prime$. 
    By definition, this means $\operatorname{Supp}(\mathbf{R}(t_x)_\ast \mathcal{O}_{\operatorname{Spec}(k_x)})\subseteq S$. 
    Hence, the image of $t_x$ on the underlying topological spaces is contained in $S$, and so $S^\prime \subseteq S$. 
    Indeed, consider the fibered square
    \begin{displaymath}
        \begin{tikzcd}
            {U \times_X \operatorname{Spec}(k_x)} & {\operatorname{Spec}(k_x)} \\
            U & {X.}
            \arrow["{s_x}", from=1-1, to=1-2]
            \arrow["{t^\prime_x}"', from=1-1, to=2-1]
            \arrow["{t_x}", from=1-2, to=2-2]
            \arrow["s"', from=2-1, to=2-2]
        \end{tikzcd}
    \end{displaymath}
    As $s$ is an \'{e}tale presentation, $s_x$ is surjective and hence $U \times_X \operatorname{Spec}(k_x)$ is nonempty. 
    Choose any $u\in |t^\prime_x| (|U \times_X \operatorname{Spec}(k_x)|)$. 
    Denote $h\colon \operatorname{Spec}(\mathcal{O}_{U,u})\to U$ the natural morphism. 
    Consider the fibered square
    \begin{displaymath}
        \begin{tikzcd}
            {\operatorname{Spec}(\mathcal{O}_{U,u}) \times_X \operatorname{Spec}(k_x)} & {\operatorname{Spec}(k_x)} \\
            {\operatorname{Spec}(\mathcal{O}_{U,u})} & {X.}
            \arrow["{g_x}", from=1-1, to=1-2]
            \arrow["{t^{\prime \prime}_x}"', from=1-1, to=2-1]
            \arrow["{t_x}", from=1-2, to=2-2]
            \arrow["{s\circ h}"', from=2-1, to=2-2]
        \end{tikzcd}
    \end{displaymath}
    By \cite[\href{https://stacks.math.columbia.edu/tag/03H4}{Tag 03H4}]{StacksProject}, the choice of $u$ implies $\operatorname{Spec}(\mathcal{O}_{U,u}) \times_X \operatorname{Spec}(k_x)$ is nonempty.
    Since $t_x$ is affine \cite[\href{https://stacks.math.columbia.edu/tag/09TF}{Tag 09TF}]{StacksProject}, base change says $t^{\prime \prime}_x$ is affine.
    If $\mathbf{L}(s\circ h)^\ast \mathbf{R}(t_x)_\ast \mathcal{O}_{\operatorname{Spec}(k_x)} \cong 0$, then flat base change implies $\mathbf{R}t^{\prime \prime}_x \mathcal{O}_{\operatorname{Spec}(\mathcal{O}_{U,u}) \times_X \operatorname{Spec}(k_x)} \cong 0$, and so \cite[Lemma 3.2]{GuisadoVillaalgordo/Lank:2026} implies $\mathcal{O}_{\operatorname{Spec}(\mathcal{O}_{U,u}) \times_X \operatorname{Spec}(k_x)} \cong 0$. 
    However, this would be absurd because $\operatorname{Spec}(\mathcal{O}_{U,u}) \times_X \operatorname{Spec}(k_x)$ is nonempty.
    Hence, $u\in \operatorname{Supp}(\mathbf{L}s^\ast \mathbf{R}(t_x)_\ast \mathcal{O}_{\operatorname{Spec}(k_x)})$.
    Choose $w\in |U \times_X \operatorname{Spec}(k_x)|$ such that $|t^\prime_x|(w)=u$. 
    Let $\eta$ be the unique point of $\operatorname{Spec}(k_x)$. 
    Thus
    \begin{displaymath}
        (|t_x\circ s_x|)(w) = |t_x|(\eta) = (|s \circ t^\prime_x|)(w)=s(u) \in \operatorname{Supp}(\mathbf{R}(t_x)_\ast \mathcal{O}_{\operatorname{Spec}(k_x)}),
    \end{displaymath}
    showing us what we want.

    Lastly, assume $S$ is specialization closed. 
    Choose any $y\in S$.
    Let $t\colon \operatorname{Spec}(k)\to X$ be a representative of $y$. 
    We want to show that $\mathbf{R}t_\ast \mathcal{O}_{\operatorname{Spec}(k)}\in D_{\operatorname{qc},S}(X)$. 
    As $X$ is quasi-separated, $t$ is affine, and so the natural morphism $t_\ast \mathcal{O}_{\operatorname{Spec}(k)} \to \mathbf{R}t_\ast \mathcal{O}_{\operatorname{Spec}(k)}$ is an isomorphism \cite[\href{https://stacks.math.columbia.edu/tag/073H}{Tags 073H} \& \href{https://stacks.math.columbia.edu/tag/09TF}{09TF}]{StacksProject}.
    Choose some $y^\prime \in \operatorname{Supp}(t_\ast \mathcal{O}_{\operatorname{Spec}(k)})$.
    Let $u\in \operatorname{Supp}(\mathbf{L}s^\ast \mathbf{R}t_\ast \mathcal{O}_{\operatorname{Spec}(k)})$ such that $|s|(u)=y^\prime$. 
    Consider the natural morphism $h\colon \operatorname{Spec}(\mathcal{O}_{U,u})\to U$.
    There exists a fibered square 
    \begin{displaymath}
        \begin{tikzcd}
            {\operatorname{Spec}(\mathcal{O}_{U,u})\times_X \operatorname{Spec}(k)} & {\operatorname{Spec}(k)} \\
            {\operatorname{Spec}(\mathcal{O}_{U,u})} & {X.}
            \arrow["{h^\prime}", from=1-1, to=1-2]
            \arrow["{t^\prime}"', from=1-1, to=2-1]
            \arrow["t", from=1-2, to=2-2]
            \arrow["{s\circ h}"', from=2-1, to=2-2]
        \end{tikzcd}
    \end{displaymath}
    Base change implies $t^\prime$ is affine. 
    Moreover, by flat base change, 
    \begin{displaymath}
        0
        \not\cong \mathbf{L}(s\circ h)^\ast \mathbf{R}t_\ast \mathcal{O}_{\operatorname{Spec}(k)} 
        \cong \mathbf{R}t^\prime_\ast \mathcal{O}_{\operatorname{Spec}(\mathcal{O}_{U,u})\times_X \operatorname{Spec}(k)}.
    \end{displaymath}
    As $t^\prime$ is affine, and affineness is preserved under base change, $\mathcal{O}_{\operatorname{Spec}(\mathcal{O}_{U,u})\times_X \operatorname{Spec}(k)}$ is nonzero.
    Hence, $\operatorname{Spec}(\mathcal{O}_{U,u})\times_X \operatorname{Spec}(k)$ is not empty.
    Choose some $q\in |\operatorname{Spec}(\mathcal{O}_{U,u})\times_X \operatorname{Spec}(k)|$. 
    Then $|t^\prime|(q) \rightsquigarrow u$. 
    Since $s\circ h$ preserves specializations, it follows $y\rightsquigarrow y^\prime$. 
    However, $S$ is specialization closed, and so $y^\prime \in S$. 
\end{proof}

\begin{lemma}
    \label{lem:Bousfield}
    Every $\otimes$-localizing subcategory $\mathcal{T}$ of $D_{\operatorname{qc}}(X)$ is an $\otimes$-Bousfield subcategory (i.e.\ the Verdier localization functor $D_{\operatorname{qc}}(X) \to D_{\operatorname{qc}}(X)/\mathcal{T}$ admits a right adjoint). 
\end{lemma}

\begin{proof}
    By \Cref{thm:point_generated_classification}, $\mathcal{T}$ is generated by a set. 
    Moreover, \cite[\href{https://stacks.math.columbia.edu/tag/09IY}{Tag 09IY}]{StacksProject} says $D_{\operatorname{qc}}(X)$ is compactly generated. 
    It follows from \cite[Remark 1.16, Theorem 1.17 \& Proposition 9.1.19]{Neeman:2001a} that a Bousfield localization exists for $\mathcal{T} \subseteq D_{\operatorname{qc}}(X)$. 
    Hence, $\mathcal{T}$ is an $\otimes$-Bousfield subcategory.
\end{proof}

\begin{remark}
    Any $\otimes$-Bousfield subcategory $\mathcal{B} \subseteq D_{\operatorname{qc}}(X)$ is $\otimes$-localizing. 
    Indeed, the Verdier localization functor $D_{\operatorname{qc}}(X) \to D_{\operatorname{qc}}(X) /\mathcal{B}$ admits a right adjoint, and hence preserves small coproducts. 
    Thus, $\otimes$-Bousfield subcategories coincide with $\otimes$-localizing subcategories for point-generated algebraic spaces.
\end{remark}

Let $E\in D_{\operatorname{qc}}(X)$ and $S\subseteq |X|$ be a Thomason subset. 
Choose an \'{e}tale presentation $s\colon U \to X$ by an affine scheme. 
Suppose $t_s \colon \operatorname{Spec}(k_s) \to X$ is a representative for each $s\in S$.
Set $\mathcal{T}$ to be 
\begin{displaymath}
    \overline{\langle \{ \mathbf{R}(t_s)_\ast \mathcal{O}_{\operatorname{Spec}(k_s)} \}_{s\in S} \rangle}.
\end{displaymath}
Denote by $h_u \colon \operatorname{Spec}(\kappa(u))\to U$ the natural morphism for each $u\in |U|$. 
Denote by $L_S$ the right adjoint to the Verdier localization $\pi\colon D_{\operatorname{qc}}(X) \to D_{\operatorname{qc}}(X)/\mathcal{T}$. 
This exists by \Cref{lem:Bousfield}. 
By \cite[Proposition 9.1.8]{Neeman:2001a}, there exists a distinguished triangle
\begin{equation}
    \label{eq:Bousfield_triangles}
    \Gamma_S (E) \to E \xrightarrow{\eta_E} (L_S \circ \pi)(E) \to \Gamma_S (E)[1]
\end{equation}
with $\eta_E$ the unit of $\pi \dashv L_S$, $\Gamma_S (E)\in \mathcal{T}$, and $(L_S \circ \pi)(E)\in \mathcal{T}^\perp$. 

\begin{lemma}
    \label{lem:coaisle_is_localizing}
    $\mathcal{T}$ is compactly generated, and $\mathcal{T}^\perp$ is an $\otimes$-localizing subcategory.
\end{lemma}

\begin{proof}
    Since $S$ is Thomason, \Cref{lem:supported_Bousfield} implies $D_{\operatorname{qc},S}(X)=\mathcal{T}$. 
    Then \cite{Hrbek/Lank/Pizzirani:2025} implies $\mathcal{T}$ corresponds to the Thomason filtration given by the rule $n\mapsto S$ for all integers $n$ (see loc.\ cit.\ for terminology).
    In particular, this means $\mathcal{T}$ is compactly generated by perfect complexes whose supports are contained in $S$. 
    Since $\mathcal{T}$ is compactly generated, $\mathcal{T}^\perp$ is closed under small coproducts. 
    Indeed, let $E_i \in \mathcal{T}^\perp$ with $i\in I$ for some set $I$. 
    For each $P\in \operatorname{Perf}(X) \cap D_{\operatorname{qc},S}(X)$, we have 
    \begin{displaymath}
        \operatorname{Hom}(P,\bigoplus_{i\in I}E_i)\cong \bigoplus_{i\in I} \operatorname{Hom}(P,E_i) \cong 0.
    \end{displaymath}
    This uses the fact that each $P$ is compact in $D_{\operatorname{qc}}(X)$. 
    This shows $\mathcal{T}^\perp$ is localizing. 
    Choose a compact generator $G$. 
    Then $\operatorname{\mathbf{R}\mathcal{H}\! \mathit{om}}(G,-)$ yields an endofunctor on $\mathcal{T}^\perp$. 
    Indeed, for each $B\in D_{\operatorname{qc},S}(X)$ and $A\in \mathcal{T}^\perp$, we have $G\otimes^{\mathbf{L}} B\in D_{\operatorname{qc},S}(X)$, and so 
    \begin{displaymath}
        \operatorname{Hom}(G\otimes^{\mathbf{L}}B,A)\cong \operatorname{Hom}(B, \operatorname{\mathbf{R}\mathcal{H}\! \mathit{om}} (G,A))\cong 0. 
    \end{displaymath}
    By \cite[\href{https://stacks.math.columbia.edu/tag/0A8A}{Tag 0A8A}]{StacksProject}, 
    \begin{displaymath}
        \operatorname{\mathbf{R}\mathcal{H}\! \mathit{om}} (G,A) 
        \cong \operatorname{\mathbb{R}\mathcal{H}\! \mathit{om}} (G,A)
    \end{displaymath}
    where the latter is the internal Hom for $D(X)$. 
    By \cite[\href{https://stacks.math.columbia.edu/tag/0A8A}{Tags 0A8A} \& \href{https://stacks.math.columbia.edu/tag/08JJ}{08JJ}]{StacksProject} we obtain an isomorphism
    \begin{displaymath}
        \operatorname{\mathbb{R}\mathcal{H}\! \mathit{om}} (G,A) 
        \cong \operatorname{\mathbb{R}\mathcal{H}\! \mathit{om}} (G,\mathcal{O}_X) \otimes^{\mathbf{L}} A.
        \cong \operatorname{\mathbf{R}\mathcal{H}\! \mathit{om}} (G,\mathcal{O}_X) \otimes^{\mathbf{L}} A.
    \end{displaymath}
    Now $\operatorname{\mathbf{R}\mathcal{H}\! \mathit{om}} (G,\mathcal{O}_X)$ is a compact generator. 
    By \Cref{lem:tensor_to_normal}, 
    \begin{displaymath}
        \overline{\langle A \rangle}_{\otimes} 
        = \overline{\langle \operatorname{\mathbf{R}\mathcal{H}\! \mathit{om}} (G,\mathcal{O}_X) \otimes^{\mathbf{L}} A \rangle}.
    \end{displaymath}
    Since $\mathcal{T}^\perp$ is localizing, we have that $M\otimes^{\mathbf{L}} A \in \mathcal{T}^\perp$ for all $M\in D_{\operatorname{qc}}(X)$.
    Thus, $\mathcal{T}^\perp$ is $\otimes$-localizing.
\end{proof}

\begin{lemma}
    \label{lem:pts_in_inverse_localizing_subcategory}
    $\mathbf{R}(h_u)_\ast \mathcal{O}_{\operatorname{Spec}(\kappa(u))} \in s^{-1} \mathcal{T}$ if, and only if, $u \in |s|^{-1}(S)$.
\end{lemma}

\begin{proof}
    If $\mathbf{R}(h_u)_\ast \mathcal{O}_{\operatorname{Spec}(\kappa(u))} \in s^{-1} \mathcal{T}$, then $\mathbf{R}(s\circ h_u)_\ast \mathcal{O}_{\operatorname{Spec}(\kappa(u))} \in \mathcal{T}$.
    By \Cref{thm:point_generated_classification}, $s\circ h_u$ represents a point which belongs to $S$. 
    Hence, $u\in |s|^{-1}(S)$. 
    Conversely, let $u\in |s|^{-1}(S)$. 
    Then $\mathbf{R}(s\circ h_u)_\ast \mathcal{O}_{\operatorname{Spec}(\kappa(u))} \in \mathcal{T}$.
    Since 
    \begin{displaymath}
        \mathbf{R}(s\circ h_u)_\ast \mathcal{O}_{\operatorname{Spec}(\kappa(u))}
        \cong \mathbf{R}s_\ast \mathbf{R}(h_u)_\ast \mathcal{O}_{\operatorname{Spec}(\kappa(u))},
    \end{displaymath}
    we are done. 
\end{proof}

\begin{lemma}
    \label{lem:aisle_to_aisle}
    $\mathbf{L}s^\ast \mathcal{T} \subseteq s^{-1}(\mathcal{T})$.
\end{lemma}

\begin{proof}
    By \Cref{lem:inverse_image_localizing}, $s^{-1}(\mathcal{T})$ is localizing. 
    Since $U$ is affine, every localizing subcategory is tensor localizing. 
    By \Cref{thm:point_generated_classification}, there exists $S^\prime \subseteq |U|$ such that 
    \begin{displaymath}
        s^{-1}\mathcal{T} 
        = \overline{\langle \{ \mathbf{R}(h_u)_\ast \mathcal{O}_{\operatorname{Spec}(\kappa(u))} \}_{u\in S^\prime}\rangle }.
    \end{displaymath}
    By \Cref{lem:pts_in_inverse_localizing_subcategory}, $S^\prime \subseteq |s|^{-1}(S)$.
    On the other hand, \Cref{lem:pts_in_inverse_localizing_subcategory} implies $|s|^{-1}(S) \subseteq S^\prime$.
    Since $|s|^{-1}(S)$ is Thomason (as it is the preimage of one), \Cref{lem:supported_Bousfield} implies $D_{\operatorname{qc},|s|^{-1}(S)}(U) = s^{-1}\mathcal{T}$.
    Since $\mathbf{L}s^\ast$ preserves support, it follows that $\mathbf{L}s^\ast D_{\operatorname{qc},S}(X) \subseteq D_{\operatorname{qc},|s|^{-1}(S)}(U)$.
\end{proof}

\begin{lemma}
    \label{lem:supported_coaisle}
    $\mathbf{R}(h_u)_\ast \mathcal{O}_{\operatorname{Spec}(\kappa(u))} \in (s^{-1} \mathcal{T})^\perp$ if, and only if, $u \in |U|\setminus |s|^{-1}(S)$.
\end{lemma}

\begin{proof}
    Applying \Cref{lem:coaisle_is_localizing} to the point-generated scheme $U$ (see \Cref{prop:ascent_descent_pt_generated}) and the Thomason subset $|s|^{-1}(S)$, it follows that $(s^{-1}\mathcal{T})^\perp$ is $\otimes$-localizing. 
    By \Cref{thm:point_generated_classification}, there exists $S^\prime\subseteq |U|$ such that 
    \begin{displaymath}
        (s^{-1}\mathcal{T} )^\perp
        = \overline{\langle \{ \mathbf{R}(h_u)_\ast \mathcal{O}_{\operatorname{Spec}(\kappa(u))} \}_{u\in S^\prime}\rangle }.
    \end{displaymath}
    By \cite[Proposition 5.10]{Alonso/Jeremias/Loureiro:2024}, 
    \begin{displaymath}
        \mathbf{R}(h_u)_\ast \mathcal{O}_{\operatorname{Spec}(\kappa(u))} \in (s^{-1}\mathcal{T} )^\perp
    \end{displaymath}
    for all $u \in |U|\setminus |s|^{-1}(S)$ (e.g.\ see the proof of loc.\ cit.\ and the notion of colocalizing subcategories).
    Thus, $S^\prime = |U|\setminus |s|^{-1}(S)$ since any $u \in S^{\prime} \setminus (|U|\setminus |s|^{-1}(S))$ belongs to $|s|^{-1}(S)$, and we would have a contradiction.
\end{proof}

\begin{lemma}
    \label{lem:coaisle_to_coaisle}
    $\mathbf{L}s^\ast \mathcal{T}^\perp \subseteq (s^{-1}\mathcal{T})^\perp$.
\end{lemma}

\begin{proof}
    By \Cref{lem:coaisle_is_localizing}, $\mathcal{T}^\perp$ is $\otimes$-localizing. 
    By \Cref{thm:point_generated_classification}, there exists $S^\prime\subseteq |X|$ such that 
    \begin{displaymath}
        \mathcal{T}^\perp = \overline{\langle \{ \mathbf{R}(t_x)_\ast \mathcal{O}_{\operatorname{Spec}(k_x)} \}_{x\in S^\prime}\rangle }.
    \end{displaymath}
    Note that $S^\prime \subseteq |X|\setminus S$. 
    Let $x\in |X|\setminus S$. 
    Choose $u_x\in |U|$ such that $|s|(u_x)=x$. 
    By \Cref{lem:pts_in_inverse_localizing_subcategory},
    $\mathbf{R}(h_u)_\ast \mathcal{O}_{\operatorname{Spec}(\kappa(u))} \in s^{-1} \mathcal{T}$ if, and only if, $u \in |s|^{-1}(S)$.
    If $\mathbf{R}(s\circ h_{u_x})_\ast  \mathcal{O}_{\operatorname{Spec}(\kappa(u_x))}\in \mathcal{T}$, then $u_x\in |s|^{-1}(S)$, which is absurd.
    By \Cref{lem:inverse_image_localizing}, we have that 
    \begin{displaymath}
        \mathbf{L}s^\ast \overline{\langle \{ \mathbf{R}(t_x)_\ast \mathcal{O}_{\operatorname{Spec}(k_x)} \}_{x\in S^\prime}\rangle }
        \subseteq \overline{\langle \{\mathbf{L}s^\ast  \mathbf{R}(t_x)_\ast \mathcal{O}_{\operatorname{Spec}(k_x)} \}_{x\in S^\prime}\rangle }.
    \end{displaymath}
    Consider the fibered square 
    \begin{displaymath}
        \begin{tikzcd}
            {U\times_X \operatorname{Spec}(k_x)} & {\operatorname{Spec}(k_x)} \\
            U & {X.}
            \arrow["{s_x}", from=1-1, to=1-2]
            \arrow["{t^\prime_x}"', from=1-1, to=2-1]
            \arrow["{t_x}", from=1-2, to=2-2]
            \arrow["s"', from=2-1, to=2-2]
        \end{tikzcd}
    \end{displaymath}
    Since $t_x$ is affine, base change implies $t^\prime_x$ is affine, and hence $U\times_X \operatorname{Spec}(k_x)$ is affine. 
    Then $\overline{\langle \mathcal{O}_{U\times_X \operatorname{Spec}(k_x)}\rangle} = D_{\operatorname{qc}}(U\times_X \operatorname{Spec}(k_x))$. 
    Applying \Cref{lem:inverse_image_localizing} and flat base change, it follows that 
    \begin{displaymath}
        \begin{aligned}
            \mathbf{R}(t^\prime_x)_\ast \overline{\langle \mathcal{O}_{U\times_X \operatorname{Spec}(k_x)}\rangle} 
            &\subseteq \overline{\langle \{\mathbf{L}s^\ast  \mathbf{R}(t_x)_\ast \mathcal{O}_{\operatorname{Spec}(k_x)} \}_{x\in S^\prime}\rangle }
            \\&\subseteq \overline{\langle \{\mathbf{R}(t^\prime_x)_\ast \mathcal{O}_{U\times_X \operatorname{Spec}(k_x)} \}_{x\in S^\prime}\rangle }.
        \end{aligned}
    \end{displaymath}
    Denote by $r_v \colon \operatorname{Spec}(\kappa(v))\to U\times_X \operatorname{Spec}(k_x)$ the natural morphism for $v\in |U\times_X \operatorname{Spec}(k_x)|$.
    Base change implies $s_x$ is an \'{e}tale presentation.
    Hence, $U\times_X \operatorname{Spec}(k_x)$ is Noetherian. 
    Then \cite[Theorem 2.8]{Neeman:1992b} and \cite[\href{https://stacks.math.columbia.edu/tag/071Q}{Tag 071Q}]{StacksProject} give
    \begin{displaymath}
        \overline{\langle \{ \mathbf{R}(r_v)_\ast \mathcal{O}_{\operatorname{Spec}(\kappa(v))} \}_{v\in |U\times_X \operatorname{Spec}(k_x)|}\rangle} = D_{\operatorname{qc}}(U\times_X \operatorname{Spec}(k_x)).
    \end{displaymath}
    Thus, 
    \begin{displaymath}
        \overline{\langle \{\mathbf{R}(t^\prime_x)_\ast \mathcal{O}_{U\times_X \operatorname{Spec}(k_x)} \}_{x\in S^\prime}\rangle } 
        = \overline{\langle \{ \{ \mathbf{R}(t^\prime_x \circ r_v)_\ast \mathcal{O}_{\operatorname{Spec}(\kappa(v))} \}_{v\in |U\times_X \operatorname{Spec}(k_x)|} \}_{x\in S^\prime}\rangle } .
    \end{displaymath}
    By \cite[\href{https://stacks.math.columbia.edu/tag/03H4}{Tag 03H4}]{StacksProject}, we see that every point above belongs to $|U|\setminus |s|^{-1}(S)$. 
    Thus, this localizing subcategory is contained in $(s^{-1}\mathcal{T})^\perp$ by \Cref{lem:supported_coaisle}. 
    This completes the proof.
\end{proof}

\begin{lemma}
    $\mathbf{L}s^\ast$ induces a $t$-exact functor $(D_{\operatorname{qc}}(X),\mathcal{T}) \to (D_{\operatorname{qc}}(U),s^{-1}\mathcal{T})$. 
    In particular, for any $E\in D_{\operatorname{qc}}(X)$, the triangle \eqref{eq:Bousfield_triangles} of $\mathbf{L}s^\ast E$ with respect to $|s|^{-1}(S)$ is isomorphic to 
    \begin{displaymath}
        \mathbf{L}s^\ast \Gamma_S (E) \to \mathbf{L}s^\ast E \xrightarrow{\mathbf{L}s^\ast \eta_E} \mathbf{L}s^\ast (L_S \circ \pi)(E) \to  \mathbf{L}s^\ast \Gamma_S (E)[1].
    \end{displaymath}
\end{lemma}

\begin{proof}
    This follows from \Cref{lem:aisle_to_aisle,lem:coaisle_to_coaisle}.
\end{proof}


\bibliographystyle{alpha}
\bibliography{mainbib}

\end{document}